\documentclass[10pt]{amsart}

    \usepackage[top=1.65in, bottom=1.65in, right=1.3in, left=1.3in]{geometry} 

\usepackage{amssymb}
\usepackage{amsthm}
\usepackage{graphics}
\usepackage{amsmath}
\usepackage{amstext}

\usepackage{fancyhdr}

\usepackage{xfrac}
\usepackage{braket}
\usepackage{mathtools}
\usepackage{stmaryrd}
\usepackage{mathrsfs}
\usepackage{extarrows}

\usepackage{microtype}

\usepackage{graphicx}
\usepackage{caption}
\usepackage{subcaption}

\usepackage{todonotes}
\usepackage{import}

\usepackage{ esint }
\usepackage{abstract}

\usepackage{tikz}
\usetikzlibrary{matrix,arrows,calc,patterns}

\title[Misiurewicz parameters and dynamical stability of polynomial-like maps]{Misiurewicz parameters and dynamical stability of polynomial-like maps of large topological degree}

\author{Fabrizio Bianchi}

\thanks{This research was partially supported by the ANR project LAMBDA, ANR-13-BS01-0002 and by the FIRB2012 grant ``Differential Geometry and Geometric Function Theory", RBFR12W1AQ 002.}

\def\h#1{\hat{#1}}
\def\b#1{\bar{#1}}
\def\t#1{\tilde{#1}}

\newcommand{\Ee}{{\mathcal E}}
\newcommand{\Ff}{{\mathcal F}}
\newcommand{\Gg}{{\mathcal G}}
\newcommand{\Hh}{{\mathcal H}}

\newcommand{\Jj}{{\mathcal J}}
\newcommand{\Kk}{{\mathcal K}}
\newcommand{\Ll}{{\mathcal L}}
\newcommand{\Mm}{{\mathcal M}}

\newcommand{\Oo}{{\mathcal O}}
\newcommand{\Pp}{{\mathcal P}}

\newcommand{\Rr}{{\mathcal R}}

\newcommand{\Uu}{{\mathcal U}}
\newcommand{\Ww}{{\mathcal W}}
\newcommand{\Vv}{{\mathcal V}}

\newcommand{\pa}[1]{\left(#1\right)}

\renewcommand{\bra}[1]{\left[#1\right]}
\newcommand{\abs}[1]{\left|#1\right|}
\newcommand{\norm}[1]{\left\|#1\right\|}

\def\lam{\lambda}
\newcommand{\mb}{\mathbb}

\newcommand{\R}{\mb R}
\newcommand{\C}{\mb C}
\newcommand{\N}{\mb N}
\newcommand{\D}{\mb D}

\renewcommand{\P}{\mb P}
\newcommand{\la}{\langle}
\newcommand{\ra}{\rangle}
\newcommand{\lla}{\left	\langle}
\newcommand{\rra}{\right\rangle}
\def\eps{\varepsilon}      

\newtheorem*{teo*}{Theorem}
\newtheorem*{keyprop*}{Key proposition}

\newtheorem{teointroletter}{Theorem}

\newtheorem{teo}{Theorem}[section]
\newtheorem{defi}[teo]{Definition}
\newtheorem{cor}[teo]{Corollary}
\newtheorem{lemma}[teo]{Lemma}
\newtheorem{prop}[teo]{Proposition}

\DeclareMathOperator{\Lip}{Lip}

\newcommand{\omcu}{\Oo (M, \C^k, \Uu)}
\newcommand{\omcv}{\Oo (M, \C^k, \Vv)}
\DeclareMathOperator{\Supp}{Supp}

\renewcommand{\h}{\widehat}

\usepackage{tikz}
\usetikzlibrary{matrix,arrows,calc,shapes,trees,positioning}

\renewcommand{\hat}{\widehat}

\def\h#1{\hat{#1}}
\def\b#1{\bar{#1}}
\def\t#1{\tilde{#1}}

\newcommand{\FF}{\mathcal{F}}
\newcommand{\MM}{\mathcal{M}}
\newcommand{\JJ}{\mathcal{J}}
\newcommand{\XX}{\mathcal{X}}

\DeclareMathOperator{\jac}{Jac}

\usepackage{enumitem}
\newlist{MA}{enumerate}{1}
\setlist[MA]{label=I.\arabic*}

\newlist{Mintro}{enumerate}{1}
\setlist[Mintro]{label=A.\arabic*}

\newlist{MB}{enumerate}{1}
\setlist[MB]{label=II.\arabic*}

\newcommand{\Crit}{{\textsf {C}  }}

\DeclareMathOperator{\Card}{Card}
\begin{document}

\maketitle

\begin{abstract}
Given a family of polynomial-like maps of large topological
degree, we
relate the presence of Misiurewicz parameters to a growth condition of the postcritical volume.
This allows us to generalize to this setting the theory of stability
and bifurcation
developed by Berteloot, Dupont and the author for endomorphisms of $\P^k$.
\end{abstract}
\vskip .4 cm 
\small{\noindent \emph{Key Words}: holomorphic dynamics, polynomial-like maps, dynamical stability, Lyapunov exponents.

\vskip .4 cm

}
\normalsize

\section{Introduction and results}

The goal of this paper is to study the dynamical stability of polynomial-like maps of large topological degree, generalizing to this setting
the theory developed in \cite{bbd2015} for endomorphisms of $\P^k$ .
Our main result
relates
bifurcation in such families
with
the volume growth of the postcritical set under iteration, generalizing the
one-dimensional equivalence between dynamical stability and normalily of the critical orbits.

The study of dynamical stability within families of holomorphic dynamical systems $f_\lam$ goes back to 
the 80s, when
Lyubich \cite{lyubich1983some} and Ma\~{n}\'e-Sad-Sullivan \cite{mane1983dynamics} independently
set the foundations of the study of holomorphic families of rational maps in dimension 1.
They proved that various natural definitions of stability (like the holomorphic motion of the
repelling cycles, or of the Julia set, or the Hausdorff continuity of this latter)
are actually equivalent
and that the stability locus is dense in the parameter space.

An important breakthrough in the field happened in 2000, when De Marco
\cite{demarco2001dynamics,demarco2003dynamics}
proved a formula relating the \emph{Lyapunov function} $L(\lam)$
of a rational map (equal to
the integral of the logarithm of the Jacobian of $f_\lam$ with
respect to the unique measure of maximal entropy for $f_\lam$)
to the critical dynamics of the family.
It turns out that the canonical 
closed and positive (1,1)-current $dd_\lam^c L$
on the parameter space
is exactly supported on the bifurcation locus, and this allowed
for the start of a measure-theoretic study of bifurcations.

In recent years, there has been a lot of activity in trying to generalize the
picture by Lyubich, Ma\~{n}\'e-Sad-Sullivan and De Marco
to higher dimension. The works by Berger, Dujardin and Lyubich \cite{dujardin2013stability, BergerDujardin}
are dedicated to the stability of H\'enon maps, while the work \cite{bbd2015}
is concerned with the case of families of endomorphisms of $\P^k$. This is the
higher-dimensional analogue of rational maps, and the bifurcation locus in this setting turns out to coincide
with the support of
$dd^c L$,
as in dimension 1. The generalization to this setting of
De Marco's formula due to Bassanelli-Berteloot \cite{BB1}
is crucial in this study.
The goal of this paper is to generalize
the description given in \cite{bbd2015}
to the setting of polynomial like maps of large topological degree,
focusing on
the relation between stability and critical dynamics.
 
Polynomial-like maps are proper holomorphic functions $f:U\to V$, where $U\Subset V\subset \C^k$ and $V$
is convex.
They
 must be thought
 of as a generalization of the endomorphisms of $\P^k$ (since lifts of these give rise to polynomial-like maps).
The dynamical study of polynomial-like maps in any dimension was undertaken by Dinh-Sibony \cite{ds_allure}.
They proved that
such systems admit a canonically defined measure of maximal entropy. Moreover, if
we restrict to polynomial-like maps
of large topological degree (see Definition \ref{defi_large_topological_degree})
this \emph{equilibrium  measure} enjoys much of the properties
of its counterpart for
endomorphisms of $\P^k$.
The main difference is the lack of a potential for this measure.

In order to state the main results of this paper
we have to give some preliminary definitions. The first one,
introduced in \cite{bbd2015} for endomorphisms of $\P^k$,
concerns \emph{Misiurewicz parameters}.
These are
the higher-dimensional analogue of the maps with a strictly preperiodic critical point in dimension 1
and are the key to understand the interplay between bifurcation and critical dynamics.

\begin{defi}\label{defi_misiurewicz}\index{Misiurewicz parameter}
Let $f_\lam$
be a holomorphic family of polynomial-like maps.
A point $\lambda_0 \in M$ is called a \emph{Misiurewicz parameter} if 
there exist
a neighbourhood
$N_{\lambda_0} \subset M$ of $\lambda_0$ and
a holomorphic map $\sigma \colon N_{\lam_0}\to \C^k$
 such that:
 \begin{enumerate}
 \item\label{defi_in_julia_rep} for every $\lambda\in N_{\lam_0}$, $\sigma(\lambda)$ is a
 repelling periodic point;
 \item\label{defi_in_julia} $\sigma(\lam_0)$ is in the Julia set $J_{\lam_0}$ of $f_{\lam_0}$;
 \item\label{defi_capta_critico} there exists an $n_0$ such that $(\lambda_0, \sigma (\lambda_0)) \in f^{n_0} (C_f)$;
 \item\label{defi_intersez} $\sigma(N_{\lambda_0}) \nsubseteq f^{n_0} (C_f)$.
 \end{enumerate}
\end{defi}

Our
main result
is
the following.

\begin{teointroletter}\label{teo_misiurewicz_intro_pl}
Let $f_\lam$
be a holomorphic family of polynomial-like maps of large topological degree. Assume that
$\lam_0$
is a Misiurewicz parameter. Then, $\lam_0 \in \Supp dd^c L$.
\end{teointroletter}

The approach to this statement in
the case of $\P^k$ \cite{bbd2015}
relies on the existence of a potential -- the 
Green function -- for the equilibrium measures. We thus
adopt a different (and more geometrical)
approach.
A crucial step in establishing the result above is
proving the following
Theorem,
which
 can be seen as a generalization of the fact that, in dimension 1,
 the bifurcation locus coincides with the non-normality locus of some critical orbit.
 
 \begin{teointroletter}\label{teo_massa_crit_intro}
Let $f_\lam$ be a holomorphic family of polynomial-like maps of large topological degree $d_t$. Then
\[dd^c L=0 \Leftrightarrow
\limsup_{n\to \infty }
\frac{1}{n}
\log
\norm{\pa{f^{n}}_* \Crit_f }
< \log c,\]
where $c$ is some constant (depending on the family) smaller than $d_t$.
\end{teointroletter}
 
The proof of this result
relies on the theory of slicing
of currents and
more precisely
on the use of
equilibrium currents, which
was initiated by Pham \cite{pham} (see also \cite{ds_cime}).

In the second part of the paper we exploit Theorem \ref{teo_misiurewicz_intro_pl}
to generalize the theory
developed in \cite{bbd2015}
to the
setting of polynomial-like maps of large topological degree.
 Consider the
 set
\[
\Jj :=\{ \gamma \in \Oo (M, \C^k) \colon \gamma(\lambda)  \in J_\lam \mbox{ for every } \lam \in M\}
\]
 The family $\pa{f_\lam}_\lam$
 naturally induces an action
 on $\Jj$, by
 $\pa{\Ff\cdot \gamma} (\lam) := f_\lam \pa{\gamma(\lam)}$.
 We denote by $\Gamma_\gamma$
 the graph in the product space of the element $\gamma\in \Jj$. The following is then the
 analogous of holomorphic motion of Julia sets in this setting (see \cite{bbd2015}).
 \begin{defi}\label{defi_lamination}\index{Equilibrium lamination}
 An \emph{equilibrium lamination} is an $\Ff$-invariant subset $\Ll$ of $\Jj$ such that
 \begin{enumerate}
\item $\Gamma_\gamma \cap \Gamma_{\gamma'} = \emptyset$ for every distinct $\gamma,\gamma' \in \Ll$,
\item $\mu_\lambda \pa{\{ \gamma(\lambda) , \gamma \in \Ll\}} =1 $ for every $\lambda \in M$, where $\mu_\lam$
is the equilibrium measure of $f_\lam$.
\item $\Gamma_\gamma$ does not meet the grand orbit of the critical set of $f$ for every $\gamma \in \Ll$,
\item the map $\Ff : \Ll \to \Ll$ is $d^k$ to $1$.
\end{enumerate}
\end{defi}

We also need
a weak
notion
of holomorphic motion for the repelling cycles in the Julia set, the \emph{repelling $J$-cycles}.
Notice that in higher dimension repelling points may be outside the Julia set
(\cite{hubbard1991superattractive,fs_examplesP2}).

\begin{defi}\label{defi_rep_asymptotically_move_hol}
We say that \emph{asymptotically all $J$-cycles move holomorphically} if there exists
a subset $\Pp = \cup_n \Pp_n \subset \Jj$ such that
\begin{enumerate}
 \item $\Card \Pp_n = d^n + o(d^n)$;
\item
every $\gamma \in \Pp_n$ is $n$-periodic; and
\item for every $M' \Subset M$, asymptotically every element of $\Pp$
is 
 repelling, i.e.,
  \[\frac{\Card \left\{\mbox{ repelling cycles in } \Pp_n\right\} }{ \Card \Pp_n }\to 1.\]
 \end{enumerate}
\end{defi}

Stability can be then characterized as follows.
\begin{teointroletter}\label{teo_equiv_intro}
 Let $f_\lam$
 be a holomorphic family of polynomial-like maps of large
 topological degree $d_t \geq 2$. Assume that the parameter space is simply connected.
 Then the following are equivalent:
 \begin{Mintro}
  \item\label{item_teo_intro_rep} asymptotically all $J$-cycles move holomorphically;
    \item\label{item_teo_intro_lam} there exists an equilibrium lamination for $f$;
  \item\label{item_teo_intro_lyap} the Lyapunov function is pluriharmonic;
\item\label{item_teo_intro_mis} there are no Misiurewicz parameters. 
 \end{Mintro}
\end{teointroletter}

The implication \ref{item_teo_intro_lyap} $\Rightarrow$ \ref{item_teo_intro_mis}
is given by Theorem \ref{teo_misiurewicz_intro_pl} and
we shall prove in detail the implication \ref{item_teo_intro_lam} $\Rightarrow$ \ref{item_teo_intro_rep} (see Section \ref{section_hol_motions}).
The proof of an analogous statement (giving the holomorphic motion of all repelling $J$-cycles) on $\P^k$ (\cite{bbd2015})
needs some assumptions on the family
to avoid possible phenomena of non-linearizability (similar to some of the hypotheses required in \cite{dujardin2013stability}
in the setting of H\'enon maps).
The proof that we present, although
it gives a slightly weaker result, has the value to apply to every family.
Our strategy is a generalization to the space of holomorphic graphs
of the method, due to Briend-Duval \cite{BriendDuval1}, to recover the
equidistribution of the repelling periodic points with respect to the equilibrum measure from the fact
that all Lyapunov exponenents are strictly positive.

For the other
implications, the strategy is essentially the same as on $\P^k$, and minor work (if any) is needed to adapt
the proofs to the current setting. We
 shall thus just focus on the differences,
 referring the reader to \cite{tesi}
 for the omitted details.
 
Finally, let us just mention that, even for families of endomorphisms of $\P^k$, the
conditions in Theorem \ref{teo_equiv_intro}
are in general not equivalent to the Hausdorff continuity of the Julia sets (see \cite{bt_desboves}).
Moreover, these conditions do not define a dense subset of the parameter space
(see \cite{bt_desboves,dujardin2016non}). These are differences with respect to the dimension 1.

\subsection*{Aknowledgements} It is pleasure to thank my advisor Fran\c{c}ois Berteloot
for introducing me to this subject,
his patient guidance
and the careful reading of this paper.
I would like to thank also Charles Favre whose comments were very useful to improve both the content and the exposition
of this paper.

\section{Families of polynomial-like maps}\label{chapter_pl}

Unless otherwise stated, all the results presented here
are due to Dinh-Sibony
(see \cite{ds_allure, ds_cime}).

\subsection{Polynomial-like maps}

The starting definition is the following.

\begin{defi}\label{defi_pl}\index{Polynomial-like map}
 A \emph{polynomial-like} map is a proper holomorphic map $g: U \to V$, where $U \Subset V$ are open subsets of $\C^k$ and $V$ is convex.
\end{defi}

A polynomial-like map is in particular
a (branched)
holomorphic covering from $U$ to $V$, of a certain degree $d_t$ (the topological degree of $g$).
We shall always assume that $d_t \geq 2$.
The \emph{filled Julia set} $K$
is the subset of $U$ given by
$K:= \bigcap_{n\geq 0} g^{-n} \pa{U}.$
Notice that $g^{-1} (K) = K = g(K)$ and thus $(K,g)$ is a well-defined dynamical system.
Lifts of endomorphisms of $\P^k$ are polynomial-like maps. Moreover,
 polynomial-like maps are stable under small perturbations.

For a polynomial-like map $g$,
the knowledge of the topological degree 
is not enough to predict the volume growth of analytic subsets.
We are thus lead to consider more
general degrees than the topological one.
In the following definition,
 we denote by $\omega$ the standard K\"ahler form on $\C^k$. Moreover, recall that the mass of a positive $(p,p)$-current $T$
 on a Borel set $X$ is given by
$\norm{T}= \int_X T\wedge \omega^{k-p}$.

  \begin{defi}
  \label{defi_dynstardeg}\index{*-dynamical degrees}
 Given a polynomial-like map $g:U\to V$,
  the \emph{*-dynamical degree of order $p$}, for $0\leq p\leq k$,
  of $g$ is given by
  \[
  d_p^* (g) := \limsup_{n\to \infty}\sup_S \norm{\pa{g^n}_*  \pa{S}}^{1/n}_W
  \]
  where $W\Subset V$ is a neighbourhood of $K$ and the sup is taken over all positive closed $(k-p,k-p)$-currents
  of mass less or equal than $1$ on a fixed neighbourhood $W' \Subset V$ of $K$.
 \end{defi}
 
 It is quite straighforward to check that this definition
 does not depend on the particular
 neighbourhoods $W$ and $W'$ chosen for the computations.
 Moreover, the following hold:
 $d^*_0=1$,
 $d_k^* = d_t$,
$d_p^* (g^m) = \pa{d_p^* (g)}^m$
and
a relation $d_p^* < d_t$
is preserved by small perturbations.

\begin{teo}
\label{teo_construction_mu}
Let $g\colon U\to V$ be a polynomial-like map and $\nu$ be a probability measure supported on $V$
which is defined by an $L^1$ form. Then $d_t^{-n}\pa{g^n}^* \nu$
converge
to a probability measure $\mu$
which does not depend on $\nu$.
Moreover, for any psh function $\phi$
on a neighbourhood of $K$ the sequence $d_t^{-n} \pa{g^n}_* \phi$
converge to $\lla\mu, \phi\rra \in \{-\infty\} \cup \R$. The measure $\mu$
is ergodic, mixing and satisfies $g^* \mu = d_t \mu$.
\end{teo}

The convergence of $d_t^{-n} \pa{g^n}_* \phi$ in Theorem \ref{teo_construction_mu}
is in $L^p_{loc}$
for every $1\leq p<\infty$
if $\lla\mu, \phi\rra$ is finite, locally uniform otherwise.

\begin{defi}\label{defi_equilibrium_measure}\index{Equilibrium measure} \index{Julia set}
The measure $\mu$ given by Theorem \ref{teo_construction_mu} is called the \emph{equilibrium measure}
of $g$. The support of $\mu$ is the \emph{Julia set} of $g$, denoted with $J_g$.
\end{defi}
The assumption on $\nu$
to be defined by a $L^1$ form can be relaxed to just asking that $\nu$
does not charge pluripolar sets.
The following Theorem ensures that $\mu$ itself does not charge the critical set of $g$.
Notice that
$\mu$ may charge proper analytic subsets. This is a difference with respect to 
the case of endomorphisms of $\P^k$.

\begin{teo}
\label{teo_log_J_mu}
 Let $f:U\to V$ be a polynomial-like map of degree $d_t$.
Then $\lla\mu, \log \abs{\jac_g}\rra \geq \frac{1}{2}\log d_t$.
\end{teo}

A consequence of Theorem \ref{teo_log_J_mu} (by Parry Theorem \cite{parry1969entropy})
is that the equilibrium
measure has entropy at least $\log d_t$.
It is thus a measure of maximal entropy (see \cite{ds_cime}).
Another important consequence
is
the existence,
by Oseledets Theorem \cite{oseledets1968multiplicative}
of the \emph{Lyapunov exponents} $\chi_i (g)$
of a polynomial-like map
with respect to the equilibrium measure $\mu$.

\begin{defi}\label{defi_Lyapunov_exp_function}\index{Lyapunov exponent}\index{Lyapunov function}
 The \emph{Lyapunov function}
 $L(g)$ is the sum
 \[L (g)= \sum \chi_i (g).\]
\end{defi}

By Oseledets
and Birkhoff Theorems, 
it follows that $L(g) = \lla \mu, \log\abs{\jac}\rra$.
By Theorem \ref{teo_log_J_mu},
we thus have
$L(g) \geq \frac{1}{2} \log d_t$ for every polynomial-like map $g$.

\subsection{Maps of large topological degree}\label{section_pl_large_top_degree}

Recall that the *-dynamical degrees
were defined in Definition \ref{defi_dynstardeg}.

\begin{defi}\label{defi_large_topological_degree}\index{Large topological degree}
A polynomial-like map is of \emph{large topological degree} if $d_{k-1}^* < d_t$.
\end{defi}

 Notice that holomorphic endomorphisms
of $\P^k$
(and thus their polynomial-like lifts)
satify the above estimate. Morever,
a small perturbation of a polynomial-like map of large topological degree still satisfy this property.

The equilibrium measure of a polynomial-like map of large topological degree
integrates psh function, and thus in particular
does not charge pluripolar sets (see \cite[Theorem 2.33]{ds_cime}).

We end this section
recalling two
equidistribution
properties (\cite{ds_allure,ds_cime})
of
the equilibrium measure
 of
 a polynomial-like map of large topological degree.

\begin{teo}\label{teo_equidistribution}
 Let $g\colon U\to V$ be a polynomial-like map of large topological degree $d_t\geq 2$.
\begin{enumerate}
\item Let $R_n$
denote the set of repelling $n$-periodic points in the Julia set $J$. Then
\[
\frac{1}{d_t^n} \sum_{a\in R_n} \delta_a  \to \mu.
\]
\item There exists a proper analytic set $\Ee$ (possibly empty) contained in the postcritical set of $g$
such that
\[
d_t^{-n} \pa{g^n}^* \delta_a = \frac{1}{d_t^n} \sum_{g^{n} (b)=a} \delta_b  \to \mu
\]
if and only if $a$ does not belong to the orbit of $\Ee$.
\end{enumerate}
\end{teo}

An important
consequence of the proof of (the second part of) Theorem
\ref{teo_equidistribution} is
that all Lyapunov exponents of a polynomial-like map of
large topological degree
are bounded below by $\frac{1}{2}\log{\frac{d_t}{d^*_{k-1}}}>0$.
This property will play a very important role in the
proof of
Theorem \ref{teo_misiurewicz_intro_pl}.
It is also crucial to
establish the existence of an equilibrium lamination from
the motion of the repelling points, see \cite{bbd2015, tesi}.

\subsection{Holomorphic families}\label{section_hol_families}

We now come to the main object of our study.

\begin{defi}\label{defi_pl_fam}
 Let M be a complex manifold and $\Uu \Subset \Vv$ be connected open subsets of $M \times \C^k$. Denote by $\pi_M$
 the standard projection $\pi_M: M \times \C^k \to M$. Suppose that 
 for every
 $\lambda  \in M$, the two sets $U_\lam := \Uu \cap \pi^{-1} (\lam)$ and $V_\lam := \Vv \cap \pi^{-1} (\lam)$
 satisfy $\emptyset \neq U_\lam \Subset V_\lam  \Subset \C^k$, that $U_\lam$ is connected and that $V_\lam$ is convex.
 Moreover, assume that $U_\lam$ and $V_\lam$ depend continuously on $\lam$ (in the sense of Hausdorff).
 A \emph{holomorphic family of polynomial-like maps}
 is a proper holomorphic map $f: \Uu  \to \Vv$ fibered over $M$, i.e., of the form
 \[
\begin{aligned}
f: & \quad \Uu & \to &   \quad \Vv &\\
& (\lambda, z) & \mapsto  &(\lambda, f_\lambda (z)).&
\end{aligned}
\]
\end{defi}

From the definition, $f$ has a well defined topological degree, that we shall always denote with $d_t$
and assume to be greater than 1.
In particular, each $f_\lam : U_\lam \to V_\lam$
is
a polynomial-like map, of degree $d_t$.
We shall denote by
$\mu_\lam$, $J_\lam$ and $K_\lam$
the equilibrium measure, the Julia set and the filled Julia set of $f_\lam$,
while $C_f$, $\jac_f$
and $\Crit_f$
will be the critical set, the determinant of the (complex)
jacobian matrix of $f$ and the integration current $dd^c \log \abs{\jac_f}$.
We may drop the subscript $f$
if no confusion arises.

 It is immediate to see that the filled Julia set $K_\lam$
 varies upper
 semicontinuously with $\lam$
 for the Hausdoff topology. This allows us,
 when dealing with local problems
 to
 assume that $V_\lam$ does not depend on $\lam$, i.e.,
 that $\Vv = M \times V$, with
 $V$ an open, convex and relatively compact subset of $\C^k$.
On the other hand, the Julia set is lower semicontinuous in $\lam$ 
for a family of maps of large topological degree (\cite{ds_cime}).

 We now recall the construction, due to Pham \cite{pham}, of an \emph{equilibrium current}
 for a family of polynomial-like maps. This is based on the following Theorem.
 We recall that a \emph{horizontal current} on a product space $M\times V$
 is a current whose support is contained in $M\times L$, where $L$ is some compact subset of $V$.
 We refer to \cite{federer1996geometric} (see also \cite{ds_cime,harvey1974characterization, siu1974analyticity})
 for the basics on slicing.
 
 \begin{teo}[Dinh-Sibony \cite{ds_geometry}, Pham \cite{pham}]\label{teo_tutto_slice_chapter1}
  Let $M$ and $V$
  two complex manifolds, of dimension $m$ and $k$.
  Let $\Rr$ be a horizontal positive closed $(k,k)$-current and $\psi$ a psh function on $M\times V$. Then
  \begin{enumerate}
   \item the slice $\lla \Rr, \pi, \lam \rra$ of $\Rr$ at $\lam$ with respect
   to the projection $\pi: M\times V \to M$ exists for every $\lam\in M$, and its mass is independend from $\lam$;
   \item the function $g_{\psi,\Rr }(\lam):= \lla \Rr, \pi, \lam \rra (\psi(\lam, \cdot))$
 is psh (or identically $-\infty$).
  \end{enumerate}
If
 $\lla \Rr, \pi,\lam_0 \rra (\psi(\lam_0, \cdot))  >-\infty$
 for some $\lam_0 \in M$, then
\begin{enumerate}
\setcounter{enumi}{2}
 \item the product $\psi \Rr$
 is well defined;
 \item for every
 $\Omega$ continuous
form of maximal degree compactly supported on $M$ we have
  \begin{equation}\label{eq_formula_slicing_chapter_1}
  \int_M \langle \Rr, \pi, \lambda \rangle (\psi) \Omega (\lam) = \langle \Rr \wedge \pi^* (\Omega), \psi \rangle.
 \end{equation}
 \end{enumerate}
 In particular, the pushforward $\pi_* (\psi\Rr)$
 is well defined and coincides with the psh function $g_{\psi,\Rr}$.
 \end{teo}

 Consider now a family of polynomial-like maps $f:\Uu\to \Vv=M\times V$.
 Let $\theta$ be a smooth
probability measure compactly supported in $V$ and consider the (positive and closed)
smooth $(k,k)$-currents on $M \times V$ defined by induction as
\begin{equation}\label{eq_sn}
\begin{cases}
S_0  = \pi_V^* (\theta)\\
S_n   := \frac{1}{d_t} f^{*} S_{n-1} = \frac{1}{d_t^n} (f^{n})^{*} S_{0}.
\end{cases}
\end{equation}
The
$S_n$' are in particular horizontal positive closed $(k,k)$-currents on $M\times V$, whose slice
mass is equal to $1$.
Moreover,
since by definition we have $\la S_0, \pi, \lam \ra = \theta$ for every $\lam \in M$, we have that
$\la S_n, \pi , \lam \ra = \frac{1}{d_t^n} \pa {f_\lam^n}^* \theta$.
In particular, since every $f_\lam : U_\lam \to V$ is a polynomial-like map, for every $\lam \in M$ we have
$\la S_n, \pi , \lam \ra \to \mu_\lam$.
The following Theorem
ensures that the limits of the sequence $S_n$
have slices equal to $\mu_\lam$.

\begin{teo}[Pham]\label{teo_pham}
 Let $f:\Uu \to \Vv$ be a holomorphic family of
 polynomial-like maps. Up to a subsequence, the forms $S_n$ defined by \eqref{eq_sn} converge to
 a positive closed $(k,k)$-current $\Ee$
 on $\Vv$, supported on $\cup_{\lam} \set{\lam}\times K_\lam$,
 such that for every $\lambda \in M$
 the slice $\langle \Ee, \pi, \lambda \rangle$ exists and is equal to $\mu_\lambda$.
\end{teo}

\begin{defi}\label{defi_equilibrium_current}\index{Equilibrium current}
An \emph{equilibrium current} for $f$
is a
positive closed current $\Ee$ on $\Vv$, supported on $\cup_\lam \set{\lam}\times K_\lam$,
such that $\lla \Ee, \pi, \lam\rra = \mu_\lam$
for every $\lam \in M$.
\end{defi}

Given an equilibrium current $\Ee$ for $f$,
the product $\log\abs{\jac} \cdot \Ee$
(and so also the intersection $\Ee \wedge \Crit_f = dd^c (\log \abs{\jac} \cdot \Ee)$)
is thus well defined (by Theorems \ref{teo_tutto_slice_chapter1} and \ref{teo_log_J_mu}).
Moreover,
the distribution $\pi_* \pa{\log\abs{\jac}\cdot \Ee}$
is represented by the (plurisubharmonic)
function
$\lam \mapsto   \lla \mu_\lam, \log\abs{\jac(\lam, \cdot)} \rra$.
Notice that, while the product $\log\abs{\jac}\cdot \Ee$
a priori depends on the particular equilibrium current $\Ee$, the pushforward by $\pi$ is independent from the
particular choice
(by \eqref{eq_formula_slicing_chapter_1}).
By Oseledets
and Birkhoff
theorems
the function
$\lam \mapsto   \lla \mu_\lam, \log \abs{\jac (\lam, \cdot)} \rra$
coincides with the
Lyapunov function
$L(\lam)$, i.e., the sum
of the Lyapunov exponents of $f_\lam$
with respect to $\mu_\lam$ (see Definition \ref{defi_Lyapunov_exp_function}).
The following definition is then well posed.

 \begin{defi}\index{Bifurcation current}\label{defi_bifurcation_current}
 Let $f:\Uu \to \Vv$
 be a holomorphic family of polynomial-like maps.
 The \emph{bifurcation current}
 of $f$ is the positive closed $(1,1)$-current on $M$
 given by
 \begin{equation}\label{formula_ddcl}
 T_{bif} :=dd^c L(\lambda) = \pi_* \left(\Crit_f\wedge \Ee\right),
 \end{equation}
  where $\Ee$ is any equilibrium current for $f$.
  \end{defi}

The following result gives an approximation of the current $u\Ee$ for $u$ psh, that we shall need in Section
\ref{section_misiurewicz_in_support}.

\begin{lemma}\label{lemma_intersezioni_tendono}
 Let $f:\Uu\to \Vv=M\times V$
 be a holomorphic family of polynomial-like maps.
 Let $\theta$
 be a smooth positive measure compactly supported on $V$. Let $S_n$ be as in \eqref{eq_sn} and $\Ee$
 be any equilibrium current for $f$.
 Let $u$ be a psh function on $M \times V$ and assume that there exists
 $\lam_0 \in M$
 such that
 $\lla \mu_{\lam_0}, u(\lam_0, \cdot) \rra > -\infty$.
 Then, for every continuous form
$\Omega$ of maximal degree and compactly supported on
$M$, we have
 \begin{equation}\label{eq_convergenza_intersezione}
\lla u S_n , \pi^{*}(\Omega) \rra \to \lla u \Ee , \pi^{*}(\Omega) \rra,
\end{equation}
where the right hand side is well defined by Theorem
\ref{teo_tutto_slice_chapter1}.
\end{lemma}

Notice that the assumption $\lla \mu_{\lam_0}, u(\lam_0, \cdot) \rra > -\infty$ at some $\lam_0$
is automatic if the family is of large topological degree, see 
\cite[Theorem 2.33]{ds_cime}.
Moreover, notice that \eqref{eq_convergenza_intersezione} holds without the need of taking the subsequence
(and the right hand side is in particular independent from the subsequence used to compute $\Ee$).
Finally, we do not need to restrict $M$ to get the statement since $\Omega$ is compactly supported.
This also follows from the compactness of
horizontal positive closed currents with bounded slice mass, see \cite{ds_geometry}.

\begin{proof}
We can suppose that $\Omega$ is a positive volume form,
since we can decompose it in its positive and negative parts $\Omega = \Omega^+ - \Omega^-$
and prove the statement for $\Omega^+$ and $\Omega^-$ separately.
Moreover, by means of a partition of unity on $M$, we can also assume that
$\Ee$ is horizontal.
 By Theorem \ref{teo_tutto_slice_chapter1},
 the product
 $u\Ee$ is well defined
 and the identity
 \eqref{eq_formula_slicing_chapter_1}
 holds with both $\Rr=\Ee$ or $S_n$ and $\psi =u$.
 So, it suffices to prove that
 \begin{equation}\label{eq_convergenza_nuova}
 \int_M \langle S_{n}, \pi, \lambda \rangle (u) \Omega (\lambda) \to \int_M \langle \Ee, \pi, \lambda \rangle (u) \Omega (\lambda).
 \end{equation}
 The assertion then follows since
 the slices of $\Ee$, and thus also the right hand side, are independent from the
 particular equilibrium current chosen.
 
 Set $\phi_{n} (\lam):=\langle S_{n}, \pi, \lambda \rangle \left(u(\lambda, \cdot)\right)$ and
 $\phi(\lambda) := \langle \Ee, \pi, \lambda \rangle \left(u(\lambda, \cdot)\right) = \langle \mu_\lambda, u(\lambda, \cdot) \rangle$.
 By
 Theorem \ref{teo_tutto_slice_chapter1},
 the $\phi_{n}$'s
 and $\phi$ are psh functions on $M$.
 Moreover, at $\lam$ fixed, we have (recalling the definition \eqref{eq_sn} of the $S_n$'s
 and the fact
 that $d^{-n}_t \pa{f_\lam^n}_* u(\lam, \cdot) \to \langle \mu_\lam, u(\lam, \cdot)\rangle$ since
 $u(\lam, \cdot)$ is psh, see Theorem \ref{teo_construction_mu})
 \[
 \begin{aligned}
 \phi_{n} (\lam) &= \lla S_{n}, \pi, \lam \rra u (\lam,\cdot)
=  \lla \frac{1}{d_t^{n}}\pa{f_\lam^{n}}^* (\theta) , u (\lam,\cdot)\rra=
  \lla \theta, \frac{1}{d_t^n} \pa{f^n_\lam}_* u(\lam,\cdot)\rra\\
 &\to
 \lla \theta, \lla \mu_\lam, u(\lam,\cdot)\rra \rra=\lla \mu_\lam, u(\lam,\cdot)\rra = \phi (\lam).
 \end{aligned}
 \]
 Since
 $u$
 is upper semicontinuous (and thus locally bounded)
 all the $\phi_n$'s are
 bounded from above. This, together with the fact that they 
 converge pointwise to $\phi$, gives that the convergence $\phi_n \to\phi$ happens
 in $L^1_{loc}$, and the assertion is proved.
\end{proof}

 \begin{cor}\label{cor_lemma_convergenza_large_top_degree}
  Let $f:\Uu \to \Vv=M\times V\subset \C^m \times \C^k$
  be a holomorphic family of polynomial-like maps.
  Let $\Ee$
  be an equilibrium current and $S_n$
  be a sequence of smooth forms
  as in \eqref{eq_sn}.
  Then for every smooth $(m-1,m-1)$-form $\Omega$
  compactly supported on $M$ we have
  \[
  \lla \Crit_f  \wedge S_n , \pi^{*}(\Omega) \rra \to \lla \Crit_f \wedge \Ee , \pi^{*}(\Omega) \rra.
  \]
 \end{cor}

 \section{Misiurewicz parameters belong to $\mbox{Supp } dd^c L$}
\label{section_misiurewicz_in_support}

In this section we prove Theorems \ref{teo_misiurewicz_intro_pl} and \ref{teo_massa_crit_intro}.
The idea will be to relate the mass of
$dd^c L$
on a given open set $\Lambda$
of the parameter space
with the growth of the mass of the currents $f^n_{*} [C]$ on the vertical set $\Vv \cap \pi_M^{-1} (\Lambda)$.
Then, we will show how the presence of a Misiurewicz parameter allows us to get the desired estimate for the growth
of the critical volume,
permitting to conclude.
We shall need the following lemma, whose proof is a simple adaptation of the one of \cite[Proposition 2.7]{ds_cime}.

\begin{lemma}\label{lemma_deg_dinamico_intorno}
 Let $f: \Uu \to \Vv$ be a holomorphic
 family of polynomial-like maps. 
 Let $\delta >d_p^* (f_{\lam_0})$.
 There exists a constant $C$ such that,
 for $\lam$ sufficiently close to $\lam_0$, we have
$
 \norm{\pa{f_\lam^n}_* (S)}_{U_\lam} \leq C \delta^n
$
 for every $n\in \N$ and every closed positive $(k-p,k-p)$-current $S$ of mass less or equal than to $1$ on $U_\lam$.
\end{lemma}

The following Theorem (which proves Theorem \ref{teo_massa_crit_intro})
gives the relation between the mass of $dd^c L$
and the growth of the mass of $\pa{f^n}_* \Crit_f$.
We recall that $\Crit_f = \log \abs{\jac_f}$
is the integration on the critical set of $f$, counting the multiplicity.
We set
$\Uu_\Lambda := \Uu\cap \pa{\Lambda \times \C^k}$ and $\Vv_\Lambda := \Vv\cap \pa{\Lambda \times \C^k}$.

\begin{teo}\label{teo_mis2_critical_volume_growth}
Let $f:\Uu\to\Vv$
be a holomorphic family of polynomial-like maps. Set
$d_{k-1}^*:= \sup_{\lam \in M} d_{k-1}^* (f_\lam)$, and assume that $d_{k-1}^*$ is finite. Let $\delta$
be greater than $d_{k-1}^*$.
Then for any open subset $\Lambda \Subset M$
there exist
positive constants $c'_1, c_1$
and $c_2$ such that, for every $n\in \N$,
\[c'_1 \norm{dd^c L}_\Lambda d_t^{n} 
\leq
\norm{\pa{f^{n}}_* \Crit_f}_{ \Uu_\Lambda}
\leq c_1 \norm{dd^c L}_\Lambda d_t^{n} + c_2
 \delta^{n}.\]
In particular, if
\[
\limsup_{n\to \infty }\frac{1}{n} \log \norm{\pa{f^{n}}_* \Crit_f }_{\Uu_\Lambda}> \log d_{k-1}^*,
\]
then $\Lambda$
intersects the bifurcation locus.
\end{teo}

Notice that
$\pa{f^n}_* \Crit_f$ actually denotes the current on $\Uu_\Lambda$ which is
 the pushforward by the proper map $f^n: f^{-n} (\Uu_\Lambda) \to \Uu_\Lambda$
 of the current $\Crit_f$ on $f^{-n} (\Uu_\Lambda)$.

\begin{proof}
The problem is local. We can thus assume that $\Vv= M \times V$, where $V$
is convex. Moreover, there
exists some open convex set $\tilde U$ such that $U_\lam \Subset \tilde U \Subset V$ for every $\lam\in M$.

Let us denote by $\omega_V$ and $\omega_M$
the standard K\"ahler forms on $\C^k$ and $\C^m$.
By abuse of notation, we denote by $\omega_V + \omega_M$
the K\"ahler form $\pi_V^* \omega_V + \pi_M^* \omega_M$
on the product space $M\times V$.
Since both $\omega_V^{k+1}$ and $\omega_M^{m+1}$ are zero,
by the definition of mass
we have
\[
\begin{aligned}
& \norm{\pa{f^n}_* \Crit_f}_{\Uu_\Lambda}
= \int_{ \Uu_\Lambda} \pa{f^n}_* \Crit_f \wedge (\omega_V + \omega_M )^{k+m-1}\\
&=
\binom{ k+m-1}{k}
\int_{ \Uu_\Lambda} \pa{f^n}_* \Crit_f \wedge \omega_V^k \wedge \omega_M^{m-1}
 +
\binom{ k+m-1}{k-1}
\int_{ \Uu_\Lambda} \pa{f^n}_* \Crit_f \wedge \omega_V^{k-1} \wedge \omega_M^{m}.
\end{aligned}
\]
We shall bound the two integrals by
means of
$\norm{dd^c L}_{\Lambda} d^n_t$ and $\delta^n$, respectively.
Let us
start
with
the first one. Let $\rho$ be a positive smooth function, compactly supported on $V$, equal to a constant $c_\rho$
on $\tilde U$
and
such
that the integral of $\rho$ is equal to 1. Notice in particular that
$\rho/c_\rho$
is equal to $1$ on $\t U$ and has total mass $1/c_\rho$.
Then
\[
\begin{aligned}
\int_{\Uu_\Lambda} \pa{f^n}_* \Crit_f \wedge \omega_V^k \wedge \omega_M^{m-1}
& \leq \frac{1}{c_\rho} \int_{\Vv_\Lambda} \pa{f^n}_* \Crit_f \wedge \pa{\pi_V^{*}\rho \cdot \omega_V^k} \wedge \omega_M^{m-1}.
\end{aligned}
\]

 Consider the smooth $(k,k)$-forms
$S_n:= \frac{\pa{f^n}^*}{d_t^n}\pa{\pi_V^{*}\rho \cdot \omega_V^k}$.
By Theorem \ref{teo_pham},
every subsequence
of $\pa{S_n}_n$ has a further subsequence $S_{n_i}$
converging
to an equilibrium current $\Ee_{(n_i)}$.
By Definition \ref{defi_bifurcation_current},
we have
$dd^c L = \pi_* (\Crit_f \wedge \Ee_{(n_i)})$.
Since
$f^* \omega_M = \omega_M$, we then have
\[
\begin{aligned}
d_t^{-n_i} \int_{\Vv_\Lambda} \pa{f^{n_i}}_* \Crit_f \wedge \pa{\pi_V^{*}\rho \cdot \omega_V^k} \wedge \omega_M^{m-1}
& = 
\int_{f^{-n_i} (\Vv_\Lambda)}  \Crit_f \wedge
S_{n_i}
\wedge \omega_M^{m-1}\\
& \to \int_{\Vv_\Lambda}  \Crit_f \wedge \Ee_{(n_i)} \wedge \omega_M^{m-1} = \norm{dd^c L}_\Lambda,
\end{aligned}
\]
where
the convergence follows from
Corollary \ref{cor_lemma_convergenza_large_top_degree}
(by means of a partition of unity on $\Lambda$).
Since the limit is independent from the subsequence, the convergence above happens without
the need of extraction (see also Lemma \ref{lemma_intersezioni_tendono}).
In particular we have
\[
 \int_{\Vv_\Lambda} \pa{f^n}_* \Crit_f \wedge \pa{\pi_V^{*}\rho \cdot \omega_V^k} \wedge \omega_M^{m-1} \leq \t c_1  \norm{dd^c L}_\Lambda d_t^n
\]
for some positive constant $\t c_1$ and the desired bound from above follows.
The bound from below is completely analogous, by means of a function $\rho$
equal to 1 on a neighbourhood of $\cup_\lam \{\lam\}\times K_\lam$.

Let us then estimate the second integral. We claim that
\begin{equation}\label{eq_int_due_claim}
\int_{\Lambda \times \t U} \pa{f^n}_* \Crit_f \wedge \omega_V^{k-1} \wedge \omega_M^{m}
=
\int_\Lambda
\norm{\pa{f_\lam^n}_* \Crit_{f_\lam}}_{\t U}
\omega_M^m.
\end{equation}
where $\Crit_{f_\lam}$
is the integration current (with multiplicity)
on the critical set of $f_\lam$.
The assertion then follows since, by
Lemma \ref{lemma_deg_dinamico_intorno},
the right hand side in \eqref{eq_int_due_claim}
is bounded by $\t c_2 \delta^n$, for some positive $\t c_2$.

Let us thus prove \eqref{eq_int_due_claim}.
By \cite[p. 124]{siu1974analyticity} and \cite[Theorem 4.3.2(7)]{federer1996geometric},
the slice $\lla
\pa{f^n}_* \Crit_f , \pi, \lam
\rra$ of $\pa{f^n}_* \Crit_f$ exists for almost every $\lam \in \Lambda$
and is given by
\[
\lla
\pa{f^n}_* \Crit_f , \pi, \lam
\rra
=
\pa{f_\lam^n}_* \lla \Crit_f, \pi, \lam \rra
=
\pa{f_\lam^n}_* { \Crit_{f_\lam}}.
\]
Since $\omega^{k-1}_V$
is smooth, this implies that the slice
of
$\pa{f^n}_* \Crit_f \wedge \omega^{k-1}_{V}$
exists for
almost
every $\lam$
and is equal to the measure
$\pa{\pa{f_\lam^n}_*  \Crit_{f_\lam}} \wedge \omega^{k-1}_{V}$.
The claim
then follows
from
\cite[Theorem 4.3.2]{federer1996geometric}
by integrating a partition of unity.
\end{proof}

Now we aim
to bound 
from below
a subsequence of
$\pa{\norm{\pa{f^n}_* \Crit_f}_{\Uu_\Lambda}}_n$
in presence of a Misiurewicz parameter.
The main tool to achieve this goal is given by the next proposition.

\begin{prop}\label{prop_mis2_costruzione_tubi}
 Let $f: \Uu \to \Vv = \D\times V$
 be a holomorphic family of polynomial-like maps of large topological degree $d_t$.
 Fix a ball $B \Subset V$ such that $B \cap J(f_0)\neq \emptyset$
 and let $\delta$ be such that
 $0 < \delta < d_t$.
 There exists a ball $A_0 \subset B$, a $N>0$ and a $\eta>0$
 such that
 $f^N$
 admits at least $\delta^N$ inverse branches defined on
 the cylinder $\D_\eta \times A_0$, with image contained in $\D_\eta \times A_0$.
\end{prop}

 In the proof of the above proposition we shall first need to
 construct a ball $A \subset B$ with the required number of inverse branches for $f_0$.
 This is done by means
 of the following general lemma.
Fix any polynomial-like map $g:U\to V$
of large topological degree.
Given any $A\subset V$,
$n\in  \N$
and $\rho>0$, denote by
$C_n (A, \rho)$
the set
\begin{equation}\label{eq_def_Cn}
C_n (A, \rho) := \left\{
h
\left|
\begin{array}{l}
h \mbox{ is an inverse branch of } g^n 
\mbox{ defined on } \b A \\
 \mbox{ and such that }  h (\b A)\subset A \mbox{ and } 
 \Lip h_{|\b A} \leq \rho
\end{array}
 \right.
 \right\}.
\end{equation}
The following
result, which is just a local
version
of \cite[Proposition 3.8]{bbd2015},
is essentially due to Briend-Duval (see \cite{BriendDuval1}).

\begin{lemma}\label{lemma_mis2_costruzione_pallina_preimmagini}
 Let $g$
 be a polynomial-like map of large topological
 degree $d_t$. Let $B$ 
 be a ball intersecting
 $J$ and
 $\rho$ a positive number. There exists a
 ball $A$
 contained in $B$ and a $\alpha>0$
 such that $ \# C_n (A,\rho) \geq
 \alpha d^n_t$,
 for every $n$
 sufficently large.
\end{lemma}

\begin{proof}[Proof of Proposition \ref{prop_mis2_costruzione_tubi}]
Let $A\subset B$
be a ball
given by an application of
Lemma \ref{lemma_mis2_costruzione_pallina_preimmagini} to the map $f_0$,
with $\rho = 1/4$.
There thus exists a $\alpha$ such that, for every sufficiently
large
$n$, the set $C_n (A, 1/4)$
defined as in \eqref{eq_def_Cn}
has at least $\alpha d_t^n$ elements.
Fix $N$ sufficiently large such that $\delta^N < \alpha d_t^N$.
Denote by
$h_i$
the elements of $C_N (A, 1/4)$
and by
$A_i$
the images
$A_i := h_i (A) \subset A$.
By definition of inverse branches, the $A_i$'s
are all disjoint and $f^N_0$
induces a biholomorphism
from every $A_i$ to $A$.

Take as $A_0$ any open ball relatively compact
in $A$ and such that $\cup_i \b A_i \Subset A_0$.
Such an $A_0$
exists
since $\cup_i \b A_i \Subset A$.
In particular, on $A_0$
the $h_i$'s are well defined, with images (compactly)
contained in the $A_i$'s.
To conclude, it suffices to find a $\eta$ such that
 these inverse branches for $f_0^N$
extend to inverse branches for
$f^N$ on $\D_\eta \times A_0$,
with images contained in $\D_\eta \times A_0$.

Define the sets $A_i^\eps$ by
\[
A_i^{\eps} := \set{z \in A_i \colon d(z, A_i^c)>\eps}.
\]
Since the $A_i$'s are finite and $f_0^N (A_i) = A$,
there exists
a $\eps_0$ such that, for every $i$,
$A_0 \Subset f_0^N (A_i^{\eps_0})$.
This implies, since $f$ is continuous and every $f_\lam$
is an open map,
that
$A_0 \Subset f_\lam^N (A_i^{\eps_0})$
for every $\lam$
sufficiently small.
Indeed, notice that $f_\lam^N (A_{i}^{\eps_0})$
is open and meets $A_0$ (for small $\lam$). We need then to show that
$\partial f_\lam^N (A_i^{\eps_0}) \cap A_0 = \emptyset$. Since $f_\lam^N$
is open, we have
$\partial f_\lam^N (A_i^{\eps_0}) \subset f_\lam^N (\partial A_i^{\eps_0})$, and the right term is close to
$f_0^N (\partial A_i^{\eps_0})$
for small $\lam$, by continuity.
But $f_0^N (\partial A_i^{\eps_0})$ is equal to $\partial f_0^N ( A_i^{\eps_0})$
since $f_0^N$
is a biholomorphism on $A_i$,
and thus
does not meet $A_0$.

We can thus consider, for $\D_\eta$
sufficiently small,
the open sets $\t A_i :=\pa{\pa{f^N}_{|\D_\eta \times A_i^{\eps_0}}}^{-1} (\D_\eta \times A_0) \Subset \D_\eta \times A_i^{\eps_0}$.
By the argument above,
for every $\lambda \in \D_\eta$
the function 
$f^N_\lam$ is a proper holomorphic map from
$\t A_{\lam,i} :=\pa{\pa{f^N_\lam}_{|A_i^{\eps_0}}}^{-1} (A_0)$ to $A_0$.
In order to conclude, we only need to check that, for $\lam$
in a neighbourhood of $0$,
the degree of
$f_\lam^N :
\t A_{\lam,i} \to A_0$
is equal to 1.
Since
$\t A_{\lam,i} \Subset A_i^{\eps_0}$, it
is enough to find $\eta$
such that the critical set of $f^N$ does not intersect $\D_\eta \times A_i^{\eps_0}$, for every $i$.
The existence of such $\eta$
follows from the Lipshitz estimate of the inverses $h_i$. Indeed, the fact that $\Lip h_i <1/4$
on $A$
implies that $\norm{\pa{df_0^N}^{-1}}^{-1} \geq 4 $ on $\cup_i A_i$.
It follows that $\norm{\pa{df_\lam^N}^{-1}}^{-1} \geq 3 $
on a neighbourhood of $\set{0}\times \cup_i A_i^{\eps_0}$
of the form $\D_\eta \times \cup_i A_i^{\eps_0}$.
In particular, the critical set of $f^N$
cannot intersect this neighbourhood, and the assertion follows.
\end{proof}

We can now prove
Theorem \ref{teo_misiurewicz_intro_pl}.

\begin{proof}[Proof of Theorem \ref{teo_misiurewicz_intro_pl}]
We shall prove that the existence of a Misiurewicz parameter implies that the
mass of $\pa{f^n}_* \Crit_f$
is asymptotically larger than $\tilde d^n$ (up to considering a subsequence),
for some $\tilde d > d^*_{k-1}$. The conclusion will then follow from Theorem \ref{teo_mis2_critical_volume_growth}.

Before starting proving the assertion, we make a few simplifications to the problem.
Let $\sigma (\lam)$ denote the repelling periodic point intersecting (but not being contained in)
  some component of $f^{n_0} (C)$ at $\lam=0$ and such that $\sigma(0) \in J_0$.
 \begin{itemize}
  \item We can suppose that $M = \D = \D_1$ and that $\lam_0 =0$.
  Doing this, we actually prove a stronger statement, i.e., that $dd^c L \neq 0$ on every complex
  disc passing through $\lambda_0$ such that $\sigma (\lam)$ is not contained in $f_\lam^{n_0} (C)$ for every $\lam$
  is the selected disc.
Moreover, we shall assume that $\Vv=\D \times V$.
   \item
  Without loss of generality, we can assume that $\sigma(\lam)$
  stays repelling for every $\lam\in \D$.
  Up to considering an iterate of $f$, we can suppose that $\sigma(\lambda)$
  is a repelling fixed point. Indeed,
  we can replace $n_0$ with $n_0 + r$, for some $0 \leq r < n(\sigma)$, where $n(\sigma)$
  is the period of $\sigma$, to ensure that now the new $n_0$ is a multiple of $n(\sigma)$.
  \item Using a local biholomorphism (change of coordinates), we can suppose that
  $\sigma(\lambda)$
  is a constant in $V$, and we can assume that this constant is $0$.
\item After possibly further rescaling, we can assume that
$f^{n_0} (C)$ intersects $\set{z=0}$
only at $\lam=0$.
  \item  We denote by
  $B$
  a small ball in $V$ centered at
  $0$.
  By taking this ball sufficiently small (and up to rescaling the parameter space), we can assume that
there exists some $b>1$
such that,
for every $\lambda \in \D=M$ and for every $z,z' \in B$,
we have
$\mbox{dist} \left(
 f_{\lambda} (z), f_{\lambda} (z') \right)
 \geq b \cdot \mbox{dist} (z,z')$.
 \end{itemize}
  
Fix 
a $\delta$ such that
$d_{k-1}^* < \delta < d_t$.
Proposition \ref{prop_mis2_costruzione_tubi}
gives the existence of
a ball $A_0\subset B$ and a $\eta$ such that the cylinder $T_0 :=\D_\eta \times A_0$
admits at least $\delta^N$
inverse branches $h_i$ for $f^N$, with images contained in $T_0$.
We explicitely notice that the images of $T_0$ under these inverse branches must be disjoint.
Up to rescaling we can still assume that $\eta=1$.

The
cylinder
$T_0$
is naturally foliated
by the ``horizontal''
holomorphic graphs $\Gamma_{\xi_z}$'s,
where $\xi_z (\lam)\equiv z$, for $z\in A_0$.
By construction, $T_0$
has at least  $\delta^{Nn}$ inverse branches for $f^{Nn}$, with images contained in $T_0$.
We denote these preimages by $T_{n,i}$, and we explicitely notice that every $T_{n,i}$
is biholomorphic to $T_0$, by the map $f^{Nn}$. In particular, $f^{Nn}$ induces a foliation
on every $T_{n,i}$, given by the preimages of the $\Gamma_{\xi_z}$'s by $f^{Nn}$.

The following elementary lemma (see \cite{tesi} for a complete proof)
shows that
there exists some $n'_0$ such that
some component $\tilde C$ of $f^{n_0 + n'_0} (C)$ intersects the graph of
every holomorphic map $\gamma: \D \to B$, and in particular every
element of the induced foliation on $T_{n,i}$.
This is a consequence of the expansivity
of $f$ on $\D \times B$
and the fact that $f^{n_0} (C) \cap\set{z=0} = (0,0)$.

\begin{lemma}\label{lemma_malefico}
 Denote by $\Gg$ the set of holomorphic maps $\gamma:\D\to B$.
 There exists an $n'_0$ such that (at least)
 one irreducible component $\t C$ of $f^{n_0 + n'_0}(C)$
 passing through $(0,0)$ 
 intersects the graph of every element of $\Gg$.
\end{lemma}

Let $n'_0$ and $\t C$ be given by
Lemma \ref{lemma_malefico}.
In particular, $\t C$
intersects every element of the induced foliations on the $T_{n,i}$'s.
Let $B_{n,i}$ denote the intersection $T_{n,i} \cap \t C$
and
set $D_{n,i} := f^{Nn}
(B_{n,i})
\subset T_0$.
The $D_{n,i}$'s
are non-empty
analytic subsets
of $T_0$
(since $f^{Nn}: T_{n,i}\to T_0$ is a biholomorphism).
Moreover, the graphs of the $\xi_z$'s intersect every $D_{n,i}$, since
their preimages in $T_{n,i}$
intersect every $B_{n,i}$.
In particular,
the projection of every $D_{n,i}$
on $V$
is equal to $A_0$.

Let us finally estimate the mass of $\pa{f^{n_0+n'_0+Nn}}_* [C]$ on $\Uu_{\D}$.
First of all,
notice that $   \pa{f^{n_0+n'_0}}_* \Crit_f \geq f_*^{n_0 + n'_0} [C] \geq [f^{n_0 + n'_0} (C)] \geq [\tilde C]$
as positive currents on $\Uu_\D$.
This implies that
\[
 \norm{f^{Nn + n_0+n'_0}_* \Crit_f}_{\Uu_\D}
  \geq \norm{f_*^{Nn} \bra{f^{n_0+n_0'} (C)}}_{\Uu_\D}
 \geq \norm{f_*^{Nn} \bra{\t C}}_{\Uu_\D} \geq   \norm{f_*^{Nn} \bra{\t C}}_{T_0}.
\]
Now, since $f^{Nn}$ gives a biholomorphism from every $T_{n,i}$ to $T_0$ and all the $T_{n,i}$'s are disjoint, we have
\[
 \norm{f_*^{Nn} \bra{\t C}}_{T_0} \geq \norm{f_*^{Nn}\pa{ \sum_i \bra{B_{n,i}}}}_{T_0}
 = \sum_i \norm{f^{Nn}_* \bra{B_{n,i}} }= \sum_i \norm{[D_{n,i}]}.
\]
By Wirtinger formula, for every $n$ and $i$ the volume
of $D_{n,i}$
is larger than the volume
of its projection $A_0$ on $V$.
Since by construction the last sum has at least $\delta^{Nn}$ terms, we have
\[
\norm{f^{n_0+n'_0+Nn}_*  \Crit_f}_{\Uu_\D} \geq \mbox{volume} (A_0) \cdot  \delta^{Nn} > \mbox{volume} (A_0) \cdot \pa{d_{k-1}^{*}}^{Nn}
\]
and the assertion follows from Theorem \ref{teo_mis2_critical_volume_growth}.
\end{proof}

 \section{Dynamical stability of polynomial-like maps}

\subsection{Equilibrium webs}

We introduce here
a notion of \emph{holomorphic motion}
for the equilibrium measures of a family of polynomial-like maps.
Consider
the space
of holomorphic
maps
\[
\Oo \pa{M,\C^k,\Vv} := \set{ \gamma \in \Oo \pa{M,\C^k} \colon  \forall \lam \in M , \gamma(\lam) \in V_\lam }
\]
endowed 
with the topology of local uniform convergence
and its subset
\[
\Oo \pa{M,\C^k,\Uu} := \set{ \gamma \in \Oo \pa{M,\C^k} \colon  \forall \lam \in M , \gamma(\lam) \in U_\lam }.
\]
By Montel Theorem, $\omcu$
is relatively compact in $\omcv$.
The map $f$ induces an action $\Ff$ from $\omcu$ to $\omcv$ given by
$\pa{\Ff\cdot \gamma} (\lam) = f_\lam \pa{\gamma(\lam)}$.
We now
restrict ouselves to
the subset $\Jj$ of $\omcu$ given by
\[
\Jj :=\{ \gamma \in \omcu \colon \gamma(\lambda)  \in J_\lam \mbox{ for every } \lam \in M\}.
\]
This is an $\Ff$-invariant
compact metric space with respect to the topology of local uniform convergence.
Thus, $\Ff$ induces a well-defined dynamical system on it.

Nothing prevents the set $\Jj$ to be actually empty, but we have the
following lemma (see \cite{tesi})
which is a consequence of the lower semicontinuity of the Julia set (see \cite{ds_cime}).

\begin{lemma}\label{lemma_open_j}
 Let $f: \Uu \to \Vv$ be a holomorphic family of polynomial-like maps of large
 topological degree
 and 
 $\rho \in \omcu$ such that $\rho(\lam)$ is $n$-periodic
 for every $\lam\in M$. Then, the set
 \[
 J_\rho := \set{\lam \in M \colon \rho(\lam) \mbox{ is $n$-periodic, repelling and belongs to } J_\lam}
 \]
 is open.
\end{lemma}

Notice
that a repelling cycle $\rho(\lam)$
can leave the Julia set (i.e., the set $J_\rho$ is not necessarily closed).
An example of this phenomenon is given in
\cite{tesi}.

We denote by
$p_\lam: \Jj\to\P^k$ the evaluation map $\gamma\mapsto \gamma(\lam)$.
The map $\Ff\colon \omcu \to \omcv$ is proper.
This follows from Montel Theorem since, for any $\lam$, $p_{\lam}$ is continuous and $f_{\lam}:U_{\lam}\to V_\lam$
is proper.
This means in particular that $\Ff$
induces a well defined notion of pushforward
from the measures on
$\omcu$
to those on $\omcv$.

\begin{defi}\index{Equilibrium web}
 Let $f: \Uu\to\Vv$ be a holomorphic family of polynomial-like maps.
 An \emph{equilibrium web} is a probability measure
 $\Mm$ on $\Jj$ such that:
 \begin{enumerate}
  \item $\Ff_* \MM = \MM$, and
  \item $\pa{p_\lam}_* \Mm =\mu_\lam$ for every $\lam \in M$, . 
 \end{enumerate}
The equilibrium measures $\mu_\lam$ \emph{move holomorphically} (over $M$) if $f$
admits an equilibrium web.
\end{defi}

Given an equilibrium web $\Mm$, we can
see the triple $\pa{\Jj, \Ff, \Mm}$ as an invariant dynamical system. Moreover,
we
can associate to
$\Mm$
the $(k,k)$-current on $\Uu$ given by
$W_\Mm := \int [\Gamma_\gamma] d\Mm (\gamma)$,
where $[\Gamma_\gamma]$
denotes the integration current
on the graph $\Gamma_\gamma \subset \Uu$ of the map
$\gamma \in \Jj$.
It is not difficult to check that $\Ww_\Mm$
is an equilibrium current for the family.

The next lemma gives some elementary properties
of the support of any equilibrium web. The proof is the same as in the case
of endomorphisms of $\P^k$ (see \cite[Lemma 2.5]{bbd2015}) and is thus omitted.

\begin{lemma}\label{lemma_property_web}
 Let $f\colon \Uu\to \Vv$
 be a holomorphic family of polynomial-like maps of degree $d_t \geq 2$. Assume that $f$
 admits an equilibrium web $\Mm$. Then
 \begin{enumerate}
  \item\label{lemma_property_web_item_1} for every $(\lam_0, z_0) \in M\times J_{\lam_0}$
  there exists an element $\gamma\in \Supp \Mm$
  such that $z_0 = \gamma (\lam)$, and
  \item\label{lemma_property_web_item_2} for every $(\lam_0, z_0) \in M\times J_{\lam_0}$ such that $z_0$
  is $n$-periodic and repelling for $f_{\lam_0}$ there exists a unique $\gamma \in \Supp \Mm$
  such that $z_0 = \gamma (\lam_0)$ and $\gamma(\lam)$ is $n$-periodic for $f_\lam$ for
  every $\lam\in M$. Moreover,
  $\gamma (\lam_0) \neq \gamma'(\lam_0)$
 for every $\gamma'\in \Supp \Mm$ different from $\gamma$.
 \end{enumerate}
\end{lemma}

The following theorem allows us to construct equilibrium webs starting from particular elements in $\omcu$.
The proof is analogous to the one on $\P^k$ (see \cite{bbd2015}) and is
based on Theorem \ref{teo_equidistribution}. We refer to \cite{tesi}
for the details in this setting. We just notice that the assumption on
the parameter space to be simply connected in the second point is needed to ensure the 
existence of the preimages.

\begin{teo}
\label{teo_construct_web}
 Let $f :\Uu \to \Vv$ be a holomorphic family of polynomial-like maps of large topological degree $d_t \geq 2$.
 \begin{enumerate}
 \item Assume that the repelling $J$-cycles of $f$
 asymptotically move holomorphically over the parameter space $M$
 and let $\pa{\rho_{n,j}}_{1 \leq j \leq N_n}$ be the elements
 of $\Jj$
 given by the motions of these cycles. Then, the equilibrium
 measures move holomorphically and any limit of
 $\pa{ \frac{1}{d^{n}_t} \sum_{j = 1}^{N_n} \delta_{\rho_{n,j}}}_n$
 is a equilibrium web.
\item Assume that the parameter space $M$
 is simply connected and that there exists $\gamma \in \omcu$
 such that the graph $\Gamma_\gamma$ does not intersect the post-critical set of $f$. Then, the equilibrium
 measures move holomorphically and any limit of
 $\pa{\frac{1}{n}\sum_{l=1}^n \frac{1}{d^{l}_t} \sum_{\Ff^l \sigma = \gamma} \delta_\sigma}_n$ is an equilibrium web.
 \end{enumerate}
 \end{teo}

\subsection{A preliminary characterization of stability}

The following Theorem
shows the equivalence
of the conditions
\ref{item_teo_intro_lyap} and \ref{item_teo_intro_mis}
in Theorem \ref{teo_equiv_intro}.

\begin{teo}
\label{teo_equiv}
 Let $f: \Uu\to \Vv$ be a holomorphic family of polynomial-like maps  of large topological degree $d_t \geq 2$.
 Then the following are equivalent:
 \begin{MA}
 \item\label{item_support}
 for every $\lam_0 \in M$ there exists a neighbourhood $M_{0} \Subset M$
 where
 the measures $\mu_\lam$
move holomorphically
and $f$ admits a equilibrium web $\Mm = \lim_n \Mm_n$ such that the graph
  of any $\gamma \in \cup_n \Supp \Mm_n$ avoids the critical set of $f$;
  \item\label{item_ddc} the function $L$ is pluriharmonic on $M$;
  \item\label{item_misiurewicz} there are no Misiurewicz parameters in $M$;
  \item\label{item_graph} for every $\lam_0 \in M$ there exists a neighbourhood $M_0 \Subset M$ and a holomorphic map
  $\gamma\in \Oo \pa{M_0, \C^k, \Uu}$ such that
 the graph of $\gamma$
 does not intersect the postcritical set of $f$.
 \end{MA}
\end{teo}

Theorem \ref{teo_construct_web}
readily proves that \ref{item_graph} imples \ref{item_support}, while
Theorem \ref{teo_misiurewicz_intro_pl} gives the implication \ref{item_ddc} $\Rightarrow$ \ref{item_misiurewicz}.
The strategy for the 
two implications \ref{item_support} $\Rightarrow$ \ref{item_ddc}
and \ref{item_misiurewicz} $\Rightarrow$ \ref{item_graph}
follows the same lines of the one on $\P^k$. In particular, for the first one 
the only small difference is how to get an estimate (H\"older in $\eps$)
for the $\mu$-measure of a $\eps$-neighbourhood
of an analytic set.
In the case of endomorphisms of $\P^k$, this follows from the H\"older continuity of the potential of the Green current.
Here we can
exploit the fact that, for every psh function $u$, the function $e^{\abs{u}}$
is integrable with respect to
the equilibrium measure of a polynomial-like map
of large topological degree (\cite[Theorem 2.34]{ds_cime}).
We
refer to \cite{tesi}
for a complete proof.

For what concerns the last missing implication
the proof
 can be reduced,
 (as in the case of endomorphisms of $\P^k$),
 by means of Hurwitz Theorem, to the proof of the existence (Theorem \ref{teo_motion_hyp_properties})
 of a hyperbolic
 set
 satisfying certain properties (see \cite{tesi}).

\begin{teo}\label{teo_motion_hyp_properties}
 Let $f: \Uu\to\Vv$ be a holomorphic family of polynomial-like maps of large topological degree $d_t$.
 Then there exists an integer $N$, a compact hyperbolic set $E_0 \subset J_0$ for $f_0^N$
 and a continuous holomorphic motion $h: B_r \times E_0 \to \C^k$ (defined on some small ball $B_r$ of radius $r$
 and centered at $0$) such that:
 \begin{enumerate}
  \item the repelling periodic points of $f_0^N$ are dense in $E_0$ 
  and $E_0$ is not contained in the postcritical set of $f_0^N$;
  \item $h_\lam (E_0) \in J_\lam$ for every $\lam \in B_r$;
  \item if $z$ is periodic repelling for $f_0^N$ then $h_\lam (z)$
  is periodic and repelling for $f_\lam^N$.
 \end{enumerate}
\end{teo}

To prove Theorem \ref{teo_motion_hyp_properties} on $\P^k$,
one needs
to ensure that a hyperbolic set of sufficiently large
entropy
cannot be contained in the postcritical set and
must, on the other hand, be contained in the Julia set.
In our setting, the analogue of the first propery is given
by Lemma \ref{lemma_entropy_mass_analytic} below, which is a direct consequence of Lemma \ref{lemma_entropia_alla_gromov}, combined
with a relative version of the Variational principle.

\begin{lemma}\label{lemma_entropia_alla_gromov}
Let $f: U\to V$ be a polynomial like map of topological degree $d_t$. Let $K$
be the filled Julia set, $X$ an
analytic subset of $V$ of dimension $p$, and $\delta_n$ be such that
$\norm{f^n_* [X]}_{U}\leq \delta_n$.
Then
\[h_t (f, \b U\cap X)= h_t (f, K\cap X) \leq \limsup_{n\to \infty} \frac{1}{n} \log \delta_n \leq d_{p}^*.\]
\end{lemma}

This Lemma is proved by following the strategy used
by Gromov \cite{gromov2003entropy} to estimate
 the topological entropy of endomorphisms of $\P^k$, and adapted by Dinh and Sibony \cite{ds_allure, ds_cime}
 to the
  polynomial-like setting. Since only minor modifications are needed, we refer to \cite{tesi}
  for a complete proof.

\begin{lemma}\label{lemma_entropy_mass_analytic}
Let $g$ be a polynomial-like map of large topological degree. Let $\nu$
be an ergodic invariant probability measure for $g$ whose metric entropy
$h_\nu$ satisfies $h_\nu > \log d^{*}_{p}$. Then, $\nu$
gives no mass to analytic subsets of dimension $\leq p$.
\end{lemma}

The second
problem (i.e., ensuring that the hyperbolic set stays inside the Julia set)
will be adressed
by means of the following Lemma.

\begin{lemma}
[see also \cite{dujardin2016non}, Lemma 2.3]
\label{lemma_motion_hyp}
Let $f:\Uu\to \Vv$ 
be a holomorphic family of polynomial-like maps with parameter space $M$.
Let $E_{0}$
be a hyperbolic set  for $f_{0}$ contained in $J_{0}$, such that repelling periodic points are dense in $E_{0}$
and $\norm{\pa{df_\lam}^{-1}}^{-1}>K>3$
on a neigbourhood of $\pa{E_0}_\tau$ in the product space.
Let $h$ be a continuous holomorphic motion of $E_{\lam_0}$
as a hyperbolic set on some ball $B\subset M$, preserving the repelling cycles.
Then $h_\lam \pa{ E_{0}}$ is contained in $J_\lam$, for $\lam$ sufficiently close to $\lam_0$.
\end{lemma}

\begin{proof}
We denote
by $\gamma_z$ the motion of a point $z\in E_0$
as part of the given holomorphic motion of the hyperbolic set.

First of all, notice that repelling points must be dense in $E_\lam$ for every $\lam$, by the continuity of the motion
and the fact that they are preserved by it.
Moreover, by Lemma \ref{lemma_open_j}, every repelling cycle stays in $J_\lam$
for $\lam$ in a neighbourhood of $0$.
It is thus enough to ensure that this neighbourhood can be taken uniform for all the cycles.
 
Since $\norm{df_\lam^{-1}}^{-1}>3$
on a neighbourhood $\pa{E_0}_\tau$ of $E_0$
in the product space, we can restrict ourselves to $\lam\in B (0,\tau)$
and so assume that $\norm{df_\lam^{-1}}^{-1}>3$ on a $\tau$ neighbourhood
of every $z\in E_\lam$, for every $\lam$.
Moreover, since the set of motions $\gamma_z$ of points
in $E_0$ is compact (by continuity),
we can assume that $\gamma_z (\lam) \in B(z, \tau/10)$ for every $z\in E_0$ and $\lam$.
Finally, by the lower semicontinuity of the Julia set (\cite{ds_cime}),
up to shrinking again the parameter space we can assume that
$J_0 \subset (J_\lam)_{\tau/10}$ for every $\lam$. These two assumptions imply that, for every $\lam$
and every $z\in E_\lam$, there exists at least a point of $J_\lam$ in the ball $B(z,\tau/2)$.
 Consider now any $n$-periodic repelling point $p_0$ in $E_\lam$ for $f_\lam$, and let
 $\set{p_i}=\set{f_\lam^i (p_0)}$ be its cycle (and thus with $p_0 = p_n$).
Fix a point $z_0 \in J_{\lam} \cap B(p_0, \tau)$.
 By hyperbolicity (and since without loss of generality we can assume that
 $\tau \leq \pa{1+\sup_{B_\tau}\norm{f_\lam}_{C^2}}^{-1}$),
 every ball $B(p_i, \tau)$ has an inverse branch for $f_\lam$
 defined on it, with image strictly 
 contained in the ball $B(p_{i-1}, \tau)$ and strictly contracting.
 This implies that there exists an inverse branch $g_0$ for $f^n_\lam$ of $B(p_0, \tau)$, strictly contracting and with image strictly
 contained in $B(p_0, \tau)$ (and containing $p_0$).
 So,
 a sequence of inverse images of $z_0$
 for $f_\lam$ must converge to $p_0$, and so $p_0\in J_\lam$. The assertion is proved.
\end{proof}

\begin{proof}[Proof of Theorem \ref{teo_motion_hyp_properties}]
 First of all, we need the hyperbolic set $E_0$.
 By Lemma \ref{lemma_mis2_costruzione_pallina_preimmagini},
 we can take a closed ball
 $A$, a constant $\rho >0$
 and a sufficiently large $N$
 such that the cardinality $N'$
 of $C_N (A,\rho)$ (see \eqref{eq_def_Cn})
  satisfies $N' \geq \pa{d^*_{k-1}}^N$ (since by assumption $d^*_{k-1} < d_t$).
  We then consider the set $E_0$ given by the intersection
  $E_0 = \cap_{k\geq 0} E_k$, where $E_k$ is given by
  \[
  E_k := \set{g_{i_1} \circ \dots \circ g_{i_k} (A) \colon (i_1, \dots, i_k) \in \set{1, \dots, N'}^{k}}
  \]
 where the $g_i$'s are the elements of $C_N (A, \rho)$.
 The set $E_0$ is then hyperbolic, and contained in $J_0$ (since $A\cap J\neq \emptyset$,
 every point in $E_0$ is accumulated by points in the Julia set). Moreover, repelling cycles (for $f_0^N$)
 are dense in $E_0$.

 Let $\Sigma : \set{1, \dots, N'}^{\N^*}$ and fix a point $z\in E_0$.
 Notice that the map $\omega: \Sigma \to E_0$ given by
 $\omega (i_1, i_2, \dots)= \lim_{k\to\infty} g_{i_1} \circ \dots \circ g_{i_k} (z)
 $
 satisfies the relation $f^N \circ \omega = \omega \circ s$,
 where $s$ denotes the left shift
 $
 (i_1,  \dots, {i_k}, \dots) \stackrel{s}{\mapsto} (i_2, \dots, i_{k+1}, \dots)$.
 We can thus pushforward with $\omega$ the uniform product measure on $\Sigma$. Since this
 is a $s$-invariant ergodic measure, its pushforward $\nu$ is an $f^N$-invariant
 ergodic measure on $E_0 \subset J_0$.
Te metric entropy of $\nu$ thus satisfies
 $h_\nu \geq \log N' \geq \log \pa{d^*_{k-1}}^N$, and this implies (by Lemma \ref{lemma_entropy_mass_analytic})
 that $\nu$ gives no mass to analytic subsets. In particular, $E_0$
 is not contained in the postcritical set of $f_0$.
 
 We need to prove the points 2 and 3.
 It is a classical result (see \cite[Appendix A.1]{bbd2015}) that $E_0$ admits a continuous holomorphic motion
 that preserves the repelling cycles, and thus 3 follows.
 The second point then follows from Lemma \ref{lemma_motion_hyp} (and the density of the repelling cycles in $E_0$).
\end{proof}

Once we have established the existence of a hyperbolic set as in Theorem \ref{teo_motion_hyp_properties},
the proof of the implication \ref{item_misiurewicz} $\Rightarrow$ \ref{item_graph}
is
the same as on $\P^k$ (see \cite{bbd2015, tesi}).

\subsection{Holomorphic motions}\label{section_hol_motions}

The following Theorem \ref{teo_equiv_2}
gives the equivalence between the conditions
 \ref{item_teo_intro_rep}
and
\ref{item_teo_intro_lam}
in Theorem \ref{teo_equiv_intro}.
We need the following definition.

\begin{defi}\label{defi_acritical}\index{Acritical web}
An equilibrium web $\Mm$ is 	\emph{acritical}  if $\Mm (\Jj_s)=0$, where
\[
\Jj_s := \set{\gamma \in \Jj \colon \Gamma_\gamma \cap \pa{\bigcup_{m,n\geq 0} f^{-m} \pa{f^n \pa{C_f}}} \neq \emptyset}.
\]
\end{defi}

\begin{teo}\label{teo_equiv_2}
 Let $f:\Uu\to\Vv$
 be a holomorphic family of polynomial-like maps of large
 topological degree $d_t \geq 2$. Assume that the parameter space is simply connected.
 Then the following are equivalent:
 \begin{MB}
  \item\label{item_cycles_move} asymptotically all $J$-cycles move holomorphically;
  \item\label{item_mjs} there exists an acritical equilibrium web $\Mm$;
  \item\label{item_laminar} there exists an equilibrium lamination for $f$.
 \end{MB}
 Moreover, if the previous conditions hold, the system admits a unique equilibrium web, which is ergodic and mixing.
\end{teo}

As we mentioned in the introduction,
we shall only show how to recover the asymptotic holomorphic motion of the repelling cycles from the two conditions \ref{item_laminar}
and \ref{item_mjs}. Indeed, the construction
of an equilibrium lamination starting from an acritical web
is literally the same as in the case of $\P^k$. The crucial point in that
proof is establishing the following backward contraction
property (Proposition  \ref{LemSF})
(which is actually used in both the implications
\ref{item_mjs}$\Rightarrow$\ref{item_laminar}$\Rightarrow$\ref{item_cycles_move}).
We set $\XX:= \JJ\setminus \Jj_s$ (notice that this is a full measure subset for an acritical web)
and let $\pa{\h \XX, \h \FF, \h \MM}$ be the \emph{natural extension} (see \cite{CFS}) of the system $\pa{\XX,\FF,\MM}$, i.e.,
the set of the histories of elements of $\XX$
\[\widehat{\gamma}:=(\cdots,\gamma_{-j},\gamma_{-j+1},\cdots,\gamma_{-1},\gamma_0,\gamma_1,\cdots),\]
where ${\mathcal F}(\gamma_{-j})=\gamma_{-j+1}$.
The map $\FF$
lifts to a bijection $\h \FF$ given by
\[
\widehat{\mathcal F}(\widehat {\gamma}):=(\cdots {\mathcal F}(\gamma_{-j}),{\mathcal F}(\gamma_{-j+1})\cdots)
\]
and thus correspond to the shift operator.
$\h \MM$ is the only measure on $\h \XX$ such that
$(\pi_{j})_\star\left(\widehat{\mathcal M}\right)={\mathcal M}$
for any projection $\pi_j:\widehat{\mathcal X}\to\widehat{\mathcal X}$ given by
$\pi_j(\widehat \gamma)=\gamma_{j}.$ When $\MM$ is ergodic (or mixing), the same is true for $\h \MM$.

Given $\gamma \in {\mathcal X}$, denote by $f_\gamma$ the injective map which is induced by $f$
on some neighbourhood of the graph
 $\Gamma_\gamma$ and by $f^{-1}_{\gamma}$ the inverse branch of $f_\gamma$,
 which is defined on some neighbourhood of $\Gamma_{{\mathcal F}(\gamma)}$.
 Given $\widehat \gamma \in \widehat{\mathcal X}$ and  $n\in \N$ we thus define
 the iterated inverse branch $f^{-n}_{\widehat \gamma}$ of $f$ along
 $\widehat \gamma$ and of depth $n$ by
 \[
 f^{-n}_{\widehat \gamma}:=f^{-1}_{\gamma_{-n}}\circ\cdots \circ f^{-1}_{\gamma_{-2}}\circ f^{-1}_{\gamma_{-1}}.
\]
Given $\eta>0$, we shall denote by $T_{M_0}(\gamma_0,\eta)$ the tubular neighbourhood
$$T_{M_0}(\gamma_0,\eta):=\{(\lam,z)\in M_0 \times\C^k\;\colon\; d (z,\gamma_0(\lam)) <\eta\}.$$

\begin{prop}\label{LemSF} 
 Let $f:\Uu \to \Vv$ be a holomorphic family of  polynomial-like maps ovr $M$
 of large topological degree $d\ge 2$ which admits an acritical
  and ergodic
equilibrium  web $\mathcal M$.
   Then (up to an iterate) there exist
a Borel subset $\widehat{\mathcal Y} \subset \widehat{\mathcal X}$
such that $\widehat{\mathcal M} (\widehat{\mathcal Y} )=1$, a measurable function 
$\widehat{\eta}
: \widehat{\mathcal Y}\to ]0,1]$ and a constant $A>0$ which satisfy the following properties.

For every $\widehat{\gamma}\in \widehat{\mathcal Y}$ and every
$n\in\N$
the iterated inverse branch  $f^{-n}_{\widehat \gamma}$ is defined on the tubular neighbourhood
$T_{U_0}(\gamma_0,\widehat{\eta}_p (\widehat \gamma))$ of $\Gamma_{\gamma_0}\cap (U_0\times\C^k)$ and 
$$f^{-n}_{\widehat \gamma}\left(T_{U_0}(\gamma_0,\widehat{\eta}_p (\widehat \gamma))\right)\subset T_{U_0}(\gamma_{-n},e^{-nA}).$$
Moreover, the map $f^{-n}_{\widehat \gamma}$ is Lipschitz
with $\Lip\;f_{\widehat{\gamma}}^{-n} \le \widehat{l}_p (\widehat{\gamma}) e^{-nA}$ where $\widehat{l}_p(\widehat{\gamma}) \ge 1$.
\end{prop}

Notice that this is essentially a local statement on  the parameter space.
The assumption on the family to be of large topological degree
is crucial here to ensure that all Lyapunov exponents are positive.
Once Proposition \ref{LemSF}
is established, the implication \ref{item_mjs}$\Rightarrow$\ref{item_laminar} follows
by an application of Poincar\'e recurrence theorem.
Moreover,
the uniqueness of the equilibrium web and its mixing behaviour also easily
follow.

The fact that either the asymptotic motion of the repelling cycles or the existence of an equilibrium lamination imply
the existence of an ergodic 
acritical web is again proved in same exact way than on $\P^k$.

\begin{prop}\label{prop_graph_then_acritical}
Let $f\colon \Uu \to \Vv$ be a holomorphic family of polynomial-like maps of
large topological degree $d_t \geq 2$.
Assume that one of the following holds:
\begin{enumerate}
\item asymptotically all repelling $J$-cycles
move holomorphically, or
\item there exists a holomorphic map $\gamma\in\omcv$
such that $\Gamma_\gamma$ does not intersect the postcritical set of $f$.
\end{enumerate}
Then $f$ admits an ergodic acritical equilibrium web.
\end{prop}

The important points in the proof are the following:
the equilibrium measure cannot charge
the postcritical set
(Theorem \ref{teo_log_J_mu})
and
we can build an equilibrium web (by means of
Theorem \ref{teo_construct_web})
satisfying the assumptions of Theorem \ref{teo_equiv}\eqref{item_support}
(which implies that the family has no Misiurewicz parameters).

We are thus left to prove that the two (equivalent) conditions \ref{item_laminar} and \ref{item_mjs}
imply the asymptotic motion of the repelling points.
We stress
that, in order to do this,
we do not need to make any further assumption on the family we are considering.
We start noticing that just the existence of any equilibrium web implies
the existence of a set $\Pp\subset \Jj$
satisfying all the properties required by Definition \ref{defi_rep_asymptotically_move_hol}
but the last one. This is an immediate consequence of Lemma \ref{lemma_property_web}.

\begin{lemma}\label{lemma_quasi_asintotic_motion}
 Let $f$ be a holomorphic family of polynomial-like maps of large topologicald degree $d_t$.
 Assume that there exists an equilibrium web $\Mm$ for $f$. Then there exists a subset $\Pp = \cup_n \Pp_n \subset \Jj$
 such that
 \begin{enumerate}
  \item $\Card \Pp_n = d^n_t + o(d^n_t)$;
  \item every element in $\Pp_n$ is $n$-periodic;
\item we have
 $\sum_{\gamma\in \Pp_n} \delta_{\gamma} \to \Mm'$,
 where $\Mm'$ is a (possibly different) equilibrium web.
  \end{enumerate}
\end{lemma}

Notice that, if the equilibrium web $\Mm$ in the statement
is acritical, by the uniqueness recalled above
we have $\Mm = \Mm'$.

\begin{proof}
Let us fix $\lam_0$
in the parameter space. Since $f_{\lam_0}$
has large topological degree, Theorem \ref{teo_equidistribution}
gives $d_t^n + o(d^n_t)$
repelling periodic points for $f_{\lam_0}$ contained in the Julia set $J_{\lam_0}$.
By Lemma \ref{lemma_property_web}\eqref{lemma_property_web_item_2}, for every such point $p$ of period $n$
there exists
an element $\gamma_p \in \Jj$ such that $\gamma_p (\lam_0) = p$ and $\Ff^n (\gamma_p) = \gamma_p$.
This gives the first two assertions of the statement. The last one follows by Theorem \ref{teo_construct_web}.
\end{proof}

In order to recover the asymptotic
motion of the repelling cycles as in Definition \ref{defi_rep_asymptotically_move_hol}, we thus just need
to prove that, on any $M'\Subset M$, \emph{asymptotically all} $\gamma_p \in \Pp_n$ given by Lemma
\ref{lemma_quasi_asintotic_motion}
are repelling. This will be
done by means of the 
following 
general lemma, which allows us to recover the existence of repelling points for a dynamical
system from the information
about backward contraction of balls along negative orbits. This can be seen as a generalization
of a classical strategy
\cite{BriendDuval1} (see also \cite{berteloot2010lyapunov}).
We keep the notations introduced
before Proposition \ref{LemSF}
regarding the natural extension of a dynamical system and the inverse branches along negative orbits.

\begin{lemma}\label{lemma_bd_metric}
 Let  $\Ff: \Kk\to \Kk$ be a
 continuous map  from a compact metric space $\Kk$ to itself.
 Assume that, for every $n$, the number of periodic repelling points of period dividing $n$
 is less than $d^n + o(d^n)$ for some integer $d\geq 2$.
 Let $\nu$ be a probability measure on $\Kk$ which is invariant, mixing  and of constant Jacobian $d$ for $\Ff$.
 
 Suppose that there exists an $\Ff$-invariant subset $\Ll\subset \Kk$ such that $\nu(\Ll)=1$ and $\Ff: \Ll \to \Ll$
 is a covering of degree $d$.
 Assume moreover that the natural extension $\pa{\widehat \Ll, \widehat \Ff, \widehat \nu}$
 of the induced system $(\Ll,\Ff,\nu)$ satisfies the following properties.
\begin{itemize}
\item[(P1)] \mbox{For every } $\h x\in \widehat \Ll
\mbox{ the inverse branch }
\Ff_{\h x}^{-n} \mbox{is defined and Lipschitz on the open ball }\\ B(x_0, \eta(\h x))
\mbox{with } \Lip \pa{\Ff_{\h x}^{-n}}\leq l(\h x) e^{-nL}, \mbox{for some positive measurable
functions } \eta \mbox{ and } l\\ \mbox{ and some positive constant } L.$
\item[(P2)]\label{item_rep_disj}
 $\forall x_0 \in \Ll, \forall N, \forall \mbox{ closed } C \subset B(x_0, \frac{1}{N})\colon \mbox{ the preimages }
 \Ff^{-n}_{\h x} (C) \mbox{ with } \eta (\h x) >\frac{1}{N} \mbox{ are disjoint}\\
  \mbox{for } n \mbox{ large enough}.$
\end{itemize}
Then,
\[
 \frac{1}{d^n} \sum_{p \in R_n} \delta_p \to \nu
\]
where $R_n$ is the set of all repelling periodic points of period (dividing) $n$.
\end{lemma}

By \emph{repelling periodic point} here we mean the following: a point $x_0$ such that, for some $n$,
$\Ff^n (x_0)=x_0$ and there exists a local inverse branch $\Hh$
for $\Ff^n$ sending $x_0$ to $x_0$ and such that $\Lip  \Hh_{x_0}<1$.

\begin{proof}

We let $\t \sigma$ be any limit value of the sequence $\sigma_n:= \frac{1}{d^n} \sum_{p \in R_n} \delta_p$. Remark that
\begin{equation}\label{eq_measure_all}
\t \sigma (\Kk) \leq \limsup_{n\to\infty} \sigma_n (\Kk) \leq \lim_{n\to\infty}\frac{d^n+o(d^n)}{d^n} =1.
\end{equation}
 For every $N\in \N$, let $\widehat \Ll_N \subset \widehat \Ll$ be defined as
 \[
 \widehat  \Ll_N =\left\{\h x \colon \eta (\h x) > \frac{1}{N} \mbox{ and } l(\h x) \leq N\right\}
 \]
 and set $\h \nu_N:= 1_{\widehat \Ll_N} \h \nu$ and $\nu_N = \pa{\pi_0}_* \h \nu_N$.
 We also set $\Ll_N := \pi_0 \pa{ \widehat \Ll_N}$.
 We are going to prove that
 \begin{equation}\label{eq_measure_greater}
  \t \sigma (A) \geq \nu_N (A) \mbox{ for every borelian } A, \quad \forall N\in \N.
 \end{equation}
 As by hypothesis $\nu_N (A)\to \nu (A)$ as $N\to\infty$, the assertion will then
 follow from \eqref{eq_measure_all} and \eqref{eq_measure_greater}.
 So we turn to prove \eqref{eq_measure_greater}. In order to do this, it
 suffices to prove the following:
 \begin{equation}\label{eq_measure_greater_N}
  \forall N\in \N, \forall \h a \in \widehat \Ll_N, \forall \mbox{ closed } C
  \subset B \pa{a_0, \frac{1}{2N}} \colon \t \sigma (C) \geq \nu_N (C).
 \end{equation}
 Indeed, given any Borelian subset $A\subset \Kk$, since $\Kk$ is compact we can find a partition of $A \cap \Ll_N$
 into finite borelian sets $A_i$, each of which contained in an open
 ball $B\pa{a^i_0, \frac{1}{3N}}$, with $\h a^i \in \Ll_N$.
 The assertion thus follows from \eqref{eq_measure_greater_N}
 since, for every $A_i$, the values $\t \sigma (A_i)$ and $\nu_N (A_i)$ are the
 suprema of the respective measures on closed subsets of
 $A_i$ (which by construction are contained in $B (a^i_0, \frac{1}{2N})$).\\
 
 In the following we thus fix a closed subset $C \subset B(a_0, \frac{1}{2N})$.
 We shall denote by $C_\delta$
 the closed $\delta$-neighbourhood of $C$ (in $\Kk$).
 Take some $\delta$ such that $\delta < \frac{1}{2N}$ and notice that,
 since $\h a \in \widehat \Ll_N$, we have
 $C_\delta \subset B(a_0, \frac{1}{N}) \subset B (a_0, {\eta (\h a)})$.
Then, according to the property (1) of the natural extension $\pa{\widehat \Ll, \widehat \Ff, \widehat \nu}$, we can define the
 set:
 \[
 \widehat R_n^\delta = \left\{\h x \in \widehat C_\delta \cap \widehat \Ll_N \colon x_0 = a_0 \mbox{ and } \Ff_{\h x}^{-n} (C_\delta) \cap C \neq \emptyset\right\}.
 \]
 Let us  denote by $S_n^\delta$ the set of preimages of $C_\delta$
 of the form $\Ff_{\h x}^{-n} (C_\delta)$ with $\h x \in \widehat R_n^\delta$,
 by the property (2) of $\pa{\widehat \Ll, \widehat \Ff, \widehat \nu}$,
 the elements of $S_n^\delta$
 are mutually disjoint for $n\geq \tilde n_0$
 (and of course $\Card\; S_n^\delta \leq d^n$).
 We claim that $\Card\; S_n^\delta$ satisfies the following two estimates:
 \begin{enumerate}
  \item\label{lemma_stima_leq} $\frac{1}{d^n} \Card\; S_n^\delta \leq \sigma_n (C_\delta)$, for $n\geq n_0 \geq \tilde n_0$,
  where $n_0$ depends only on $C$ and $\delta$;
  \item\label{lemma_stima_geq} $\frac{1}{d^n} \Card\; S_n^\delta \, \nu (C_\delta) \geq \h \nu \pa{\widehat \Ff^{-n} \pa{\widehat C_\delta \cap \widehat \Ll_N} \cap \widehat C}$.
 \end{enumerate}

Before proving the estimates (\ref{lemma_stima_leq}) and (\ref{lemma_stima_geq}), let us show how \eqref{eq_measure_greater_N} follows from them.
Combining (\ref{lemma_stima_leq}) and (\ref{lemma_stima_geq}) we get
$
\h \nu \pa{\widehat \Ff^{-n} \pa{\widehat C_\delta \cap \widehat \Ll_N} \cap \widehat C} \leq \nu (C_ \delta) \sigma_n (C_\delta)
$
and, since $\h \nu$ is mixing,
letting $n\to\infty$ on a subsequence such that $\sigma_{n_i} \to \tilde \sigma$
we find
$
\h \nu \pa{\widehat C_\delta \cap \widehat \Ll_N} \h \nu \pa{\widehat C} \leq \nu (C_\delta) \t \sigma (C_\delta).
$
Since the left hand side is equal to $\nu_N \pa{C_\delta} \nu \pa{C}$ (and $C$ is closed),
\eqref{eq_measure_greater_N} follows letting $\delta \to 0$.\\

We are thus left to proving the inequalities (\ref{lemma_stima_leq}) and (\ref{lemma_stima_geq}) above. We shall see that the first one follows
from the Lipschitz estimate on $\Ff_{\h x}^{-n}$, while the second is a consequence of the fact that $\nu$ is of constant Jacobian
(i.e., for every borelian set $A\subset \Kk$ on which $\Ff$ is injective we have
 $d \nu (A) = \nu (\Ff(A))$).

Let us start with (\ref{lemma_stima_leq}). We have to find an integer  $n_0$ such that, for $n\geq n_0$, the neighbourhood $C_\delta$
contains al least $\Card\; S_n^\delta$ repelling periodic points for $\Ff$.
Take any $\h x \in \widehat R_n^\delta$. Since $\widehat R_n^\delta\subset \widehat\Ll_N$, one has $\eta (\h x) \geq \frac{1}{N}$ and $l(\h x)\leq N$. This means that
$\Ff_{\h x}^{-n}$ is well defined on $C_\delta \subset B (a_0, \frac{1}{N})$ and that
$\mbox{Diam } \Ff_{\h x}^{-n} (C_\delta) \leq \frac{1}{N} \Lip \Ff_{\h x}^{-n} \leq \frac{1}{N} N e^{-nL}= e^{-nL}$.
Let us now take  $n_0$ such that $3 e^{-n_0 L}< \delta$.
Since, by definition of $\widehat R_n^\delta$, we have that
$\Ff_{\h x}^{-n}(C_\delta)$ intersects $C $ and $C\subset C_\delta$,  
it follows that $\Ff_{\h x}^{-n}(C_\delta) \subset C_\delta$ for every $\h x \in \h R_n^\delta$, with  $n\geq n_0$.
So, since $C_\delta$ is itself a compact metric space and $\Ff_{\h x}^{-n}$ is stricly contracting on it
(the condition $3 e^{-n_0 L}< \delta < \frac{1}{2N}$
also implies that
$\Lip \Ff_{\h x}^{-n} < 1$ for $n\geq n_0$), we
find a (unique) fixed point for it in $C_\delta$.
Since the elements of $S_n^\delta$
are disjoint, we have found at least $\Card\; S_n^\delta$
periodic points (whose period divides $n$) for $\Ff$ in $C_\delta$,
which must be repelling
by the Lipschitz estimate of the local inverse,
and so (\ref{lemma_stima_leq}) is proved.

For the second inequality,
we have
\[
 \begin{aligned}
\h \nu \pa{\widehat \Ff^{-n} \pa{\widehat C_\delta \cap \widehat \Ll_N} \cap \widehat C}
& \le \nu \pa{\pi \pa{\widehat \Ff^{-n} \pa{\widehat C_\delta \cap \widehat \Ll_N} \cap \widehat C}}
 \leq \nu \pa{ \bigcup_{\h x \in \widehat R_n^\delta} \pa{\Ff_{\h x}^{-n} (C_\delta)} \cap C} \\
\leq \nu \pa{ \bigcup_{\h x \in \widehat R_n^\delta} \pa{\Ff_{\h x}^{-n} (C_\delta)} }
& = \sum_{\Ff_{\h x}^{-n} (C_\delta)\in S_n^\delta} \nu \pa{\Ff_{\h x}^{-n} (C_\delta)}
 = \sum_{\Ff_{\h x}^{-n} (C_\delta)\in S_n^\delta} \frac{1}{d^n} \nu (C_\delta)\\
= \frac{1}{d^n} \left(\Card\; S_n^\delta\right) \nu(C_\delta)
\end{aligned}
\]
where the second  equality
follows from the fact that $\nu$ is of constant Jacobian.
\end{proof}

We can now show how conditions \ref{item_mjs} and \ref{item_laminar}
(which we recall are equivalent)
imply condition \ref{item_cycles_move} in Theorem \ref{teo_equiv_2}. 

\begin{teo}
 Let $f:\Uu \to \Vv$ be a holomorphic family of polynomial-like maps, of degree $d_t \geq 2$.
 Assume that
 there exist an acritical equilibrium web $\Mm$
 and an equilibrum lamination $\Ll$
 for $f$.
 Then,
 there exists a subset $\Pp = \cup_n \Pp_n \subset \Jj$,
 such that
 \begin{enumerate}
 \item $\Card \Pp_n = d^n + o(d^n)$;
\item
every $\gamma \in \Pp_n$ is $n$-periodic; and
\item $\forall M' \Subset M$, asymptotically every element of $\Pp$
is 
 repelling: $\frac{\Card \{\mbox{ repelling cycles in } \Pp_n\}}{ \Card \Pp_n }\to 1$.
 \end{enumerate}
 
 Moreover,
 $\sum_{\Pp_n} \delta_\gamma \to \Mm$.
\end{teo}

The need to restrict to compact subsets of $M$
is due to the fact that
the construction of the equilibrium lamination is essentially local
(see Proposition \ref{LemSF}).
Thus, the
 assumptions of Lemma \ref{lemma_bd_metric}
are satisfied on relatively compact subsets of $M$.

 \begin{proof}
  We consider the set $\Pp=\cup_n \Pp_n \subset \Jj$
  given by Lemma \ref{lemma_quasi_asintotic_motion}. We just need to prove the third assertion. We thus fix $M'\Subset M$
  and consider the compact metric space $\Oo (M', \b \Uu, \C^k)$.
  By Proposition \ref{LemSF}
and the implication \ref{item_cycles_move} $\Rightarrow$ \ref{item_laminar} of Theorem \ref{teo_equiv_2}
  all the assumptions of Lemma \ref{lemma_bd_metric} are satisfied by the system
  $(\Oo (M', \b \Uu, \C^k), \Ff, \Mm)$, with $\Ll$ any equilibrium lamination for the system.
  The assumption (P2)
 is verified since
 this is true at any fixed parameter.
  The statement follows from the following two assertions:
  \begin{enumerate}
   \item for every repelling periodic $\gamma \in R_n$ given by Lemma \ref{lemma_bd_metric}, the point
   $\gamma(\lam)$ is repelling for every $\lam\in M'$; and
   \item asymptotically all elements of $R_n$
   coincide
   with elements of $\Pp_n$.
  \end{enumerate}
The first point is a consequence of the Lipschitz estimate
of the local inverse of $\Ff^n$ at the points of $R_n$ (since the Lipschitz constant of $\Ff^{-n}$ dominates the Lipschitz constant
of $f_\lam^{-n}$, for every $\lam$),
the second of the fact that both $\Pp_n$
and $R_n$ have cardinality $d_t^n + o(d_t^n)$ and, at every $\lam$, the number of $n$-periodic points is
$d_t^n$.
 \end{proof}

 \subsection{Proof of Theorem \ref{teo_equiv_intro}
 }\label{section_equivalence_AB}
 
 In this section we show that the conditions stated in the Theorems \ref{teo_equiv} and \ref{teo_equiv_2}
 are all equivalent. This completes the proof of Theorem \ref{teo_equiv_intro}.

\begin{teo}\label{teo_equivalenti}
Let $f:\Uu\to\Vv$
 be a holomorphic family of polynomial-like maps  of large
 topological degree $d_t \geq 2$. Assume that the parameter space is simply connected. Then conditions \ref{item_support} -- \ref{item_graph}
 of Theorem
 \ref{teo_equiv} and \ref{item_cycles_move} -- \ref{item_laminar}
 of Theorem \ref{teo_equiv_2}
 are all equivalent.
\end{teo}

 It is immediate to see, by the definition of an equilibrium lamination, that condition \ref{item_laminar}
 (the existence of a lamination)
 implies condition \ref{item_graph} (existence of a graph avoiding the postcritical set), since any element in the lamination
 satisfies
 the desired property.
 Viceversa, by Proposition \ref{prop_graph_then_acritical},
 we see that condition \ref{item_graph}
 directly implies
 a local version of Theorem \ref{teo_equiv}.
 Using the uniqueness of the equilibrium lamination, we can nevertheless recover that the conditions in Theorem
 \ref{teo_equiv} imply the ones in Theorem \ref{teo_equiv_2} on all the parameter space. This is done in the
 following proposition.
 
 \begin{prop}\label{prop_global_laminar}
  Let $f:\Uu \to \Vv$ be a holomorphic family of polynomial-like maps of large topological degree $d_t \geq 2$. Assume that
  the parameter space $M$ is simply connected and
  that every point $\lam_0 \in M$ has a neighbourhood where the system admits an equilibrium lamination. Then $f$
  admits an equilibrium lamination on all the parameter space.
 \end{prop}

 In particular, if condition \ref{item_graph} holds, the assumptions of Proposition \ref{prop_global_laminar}
 are satisfied (by Proposition \ref{prop_graph_then_acritical} and Theorem
 \ref{teo_equiv_2})
 and thus condition \ref{item_laminar}
 holds, too. This complete the proof of Theorem \ref{teo_equivalenti}.
 
 \begin{proof}
 Consider a countable cover $\left\{B_n\right\}$ by open balls of the parameter space $M$, with the property that
  on every $B_n$ the system admits an equilibrium lamination $\Ll_n$. In particular, on every $B_n$
  the restricted system admits an acritical web.
Consider two intersecting balls $B_1$ and $B_2$. By the uniqueness of the equilibrium web on the intersection
(which is simply connected), both the corresponding webs induce the same one on $B_1 \cap B_2$.
  By analytic continuation, and up to
  removing a zero-measure (for the web on the intersection)
  subset of graphs from the laminations $\Ll_1$ and $\Ll_2$
  (and all their images and preimages, which are always
  of measure zero),
  we obtain a set of holomorphic graphs, defined on all of $B_1 \cup B_2$, that satisfy all the properties
  required in Definition \ref{defi_lamination}, thus giving an equilibrium lamination there.
  The assertion follows repeating the argument, since the cover is countable
  and $M$ is simply connected (and thus we do not have holonomy problems when glueing the laminations).
 \end{proof}

\bibliography{bib_pl.bib}{}
\bibliographystyle{alpha}

\end{document}